\documentclass[a4paper,11pt]{article}

\usepackage{a4wide}
\usepackage{amsmath}
\usepackage{amssymb, amsfonts}
\usepackage{amsthm}
\usepackage{color}
\usepackage{enumerate}
\usepackage{verbatim}

\usepackage{graphicx}
\usepackage{tikz}

\newcommand  {\bb}[1]{\mathbb{#1}}
\newcommand  {\bo}[1]{\boldsymbol{#1}}

\renewcommand{\ge}{\geqslant}
\renewcommand{\geq}{\geqslant}
\renewcommand{\le}{\leqslant}
\renewcommand{\leq}{\leqslant}

\newcommand  {\e}{\varepsilon}
\newcommand  {\eps}{\varepsilon}
\newcommand  {\lam}{\lambda}
\renewcommand{\phi}{\varphi}

\renewcommand{\bar}{\overline}
\renewcommand{\hat}{\widehat}

\newcommand  {\E}{\bb{E}}
\newcommand  {\Ex}{\bb{E}}
\newcommand  {\F}{\mathcal{F}}
\newcommand  {\cF}{\mathcal{F}}
\newcommand  {\cG}{\mathcal{G}}
\newcommand  {\cT}{\mathcal{T}}
\renewcommand{\P}{\bb{P}}

\newcommand  {\<}{\langle}
\renewcommand{\>}{\rangle}
\newcommand  {\all}{\forall \ }

\newcommand  {\defeq}{\stackrel{\mathrm{def}}{=}}
\newcommand  {\R}{\bb{R}}

\DeclareMathOperator{\diam}{diam}

\DeclareMathOperator{\Span}{Span}

\newtheorem{thm}{Theorem}[section]
\newtheorem{lem}[thm]{Lemma}
\newtheorem{cor}[thm]{Corollary}
\newtheorem{prop}[thm]{Proposition}

\newtheorem{rmk}[thm]{Remark}

\begin{document}

\title{The shape of multidimensional Brunet--Derrida particle systems}

\author{
Nathana\"{e}l Berestycki
 \footnote{Statistical Laboratory, University of Cambridge.}
 \footnote{Research supported in part by EPSRC grants EP/GO55068/1 and EP/I03372X/1.}
\and Lee Zhuo Zhao $^*$
}

\date{\today}

\maketitle

\begin{abstract}
We introduce particle systems in one or more dimensions in which particles perform branching Brownian motion and the population size is kept constant equal to $N > 1$, through the following selection mechanism: at all times only the $N$ fittest particles survive, while all the other particles are removed. Fitness is measured with respect to some given score function $s:\R^d \to \R$. For some choices of the function $s$, it is proved that the cloud of particles travels at positive speed in some possibly random direction. In the case where $s$ is linear, we show under some assumptions on the initial configuration that the shape of the cloud scales like $\log N$ in the direction parallel to motion but at least $c(\log N)^{3/2}$ in the orthogonal direction for some $c > 0$. We conjecture that the exponent $3/2$ is sharp. This result is equivalent to the following result of independent interest: in one-dimensional systems, the genealogical time is greater than $c(\log N)^3$, thereby contributing a step towards the original predictions of Brunet and Derrida. We discuss several open problems and also explain how our results can be viewed as a rigorous justification of Weismann's arguments for the role of recombination in population genetics.
\end{abstract}


\section{Introduction}

\subsection{Main results}

Let $d \geq 1$ and let $s : \R^d \mapsto \R$ denote a fixed function, which we will refer to as the \emph{score} or \emph{fitness} function in what follows. We consider the following system of $N$ particles in $\bb R^d$, $(X_1(t),\ldots, X_N(t))$  defined informally by the following two rules:

\begin{enumerate}[$\bullet$]
 \item Each particle $X_i$ follows the trajectory of an independent Brownian motion. 
 \item In addition each particle undergoes binary branching at rate 1. After each branching event, we remove from the population the particle $i$ with minimal score, i.e., $\min_{1 \leq i \leq n} s(X_i(t))$.
\end{enumerate}
 
Note in particular that the population size stays constant (equal to $N$) throughout time. Unless otherwise specified, we will always order particles $X_1(t), \ldots, X_N(t)$ by decreasing fitness, i.e., so that 
\begin{equation}
 s(X_1(t)) \geq \ldots \geq s(X_N(t))
\end{equation}
with arbitrary choice in case of a tie.
 

This process can be seen as a multi-dimensional generalisation of the model of branching Brownian motion with selection in $\bb R$ introduced by Brunet, Derrida, Mueller and Munier \cite{BDMM1, BDMM2}. This is the model which arises as a particular case of the above description with $d=1$ and $s(x) = x$.

The motivation for this process was the study of the effect of natural selection on the genealogy of a population. Using nonrigorous methods, Brunet et al. made several striking predictions, which we summarise below. Ordering the particles from right to left (so $X_1(t) \geq \ldots X_N(t)$):

\begin{enumerate}[(i)]
 \item Then for fixed $N$, $\lim_{t\to \infty} (X_1(t) /t) = \lim_{t\to \infty} (X_N(t)/t) \defeq v_N$, almost surely, where $v_N$ is a deterministic constant.
 \item As $N\to \infty$, $v_N = v_\infty - c/(\log N)^2 + o((\log N)^{-2})$, where $v_\infty$ is the speed of the rightmost particle in a free branching Brownian motion (or free branching random walk if time is discrete), and $c$ is an explicit constant.
 \item Finally, the genealogical time scale for this population is $(\log N)^3$. More precisely, the genealogy of an arbitrary sample of the population, resealed by $(\log N)^3$, converges to the Bolthausen--Sznitman coalescent (see for instance \cite{Ensaios} for definitions and more discussion about this problem).
\end{enumerate}

The arguments of Brunet et al. \cite{BDMM1, BDMM2} relied on a nonrigorous analogy with noisy Fisher--Kolmogorov--Petrovskii--Piskounov (FKPP) equation
\begin{equation}\label{E:FKPP}
 \frac{\partial u}{\partial t} = \frac12 \frac{\partial^2 u}{\partial x^2} + u(1 - u),
\end{equation}
and relied strongly on ideas developed earlier by Brunet and Derrida \cite{BD1, BD2, BD3} on the effect of noise on such an equation. For this reason this process is sometimes known as the Brunet--Derrida particle system. From a rigorous point of view,  proofs of  (i) and (ii) can be found in the paper of B\'erard and Gou\'er\'e \cite{BG}, while a rigorous proof of (iii) can be found in \cite{BBS} for a closely related model. However (iii) remains open for the original Brunet--Derrida process, though exciting progress in this direction has been achieved recently by Maillard \cite{Maillard}.

\medskip The main goal of this paper is to study geometric properties of the $d$-dimensional systems and to partly resolve prediction 3 above in the case $d = 1$. We start with our results in $d$ dimensions. Our results are valid in two particular cases:

\begin{enumerate}[({Case} A)]
 \item Euclidean case: $s(x_1, \ldots, x_d) = \sqrt{x_1^2 + \ldots + x_d^2}$.
 \item Linear case: for some vector $\lam \in \R^d$, $s(x) = \< \lam, x\>$. 
\end{enumerate}

See Figure \ref{F:modulus} for two realisations of the process in the Euclidean case (case A). The linear case (case B) in the two dimensional case ($d=2$) is particularly relevant from the point of view of applications, since it is reasonable to assume that for diploid populations, the total fitness of any given individual is a linear combination of the fitnesses of each of her alleles. (In this interpretation we thus view each coordinate as the fitness of the allele on the corresponding chromosome, and so the `spatial' position has nothing to do with the geographical position of that individual in space. See below for further discussion about the biological relevance of our results.)

\begin{figure}
\centering
\includegraphics[width=0.495\textwidth]{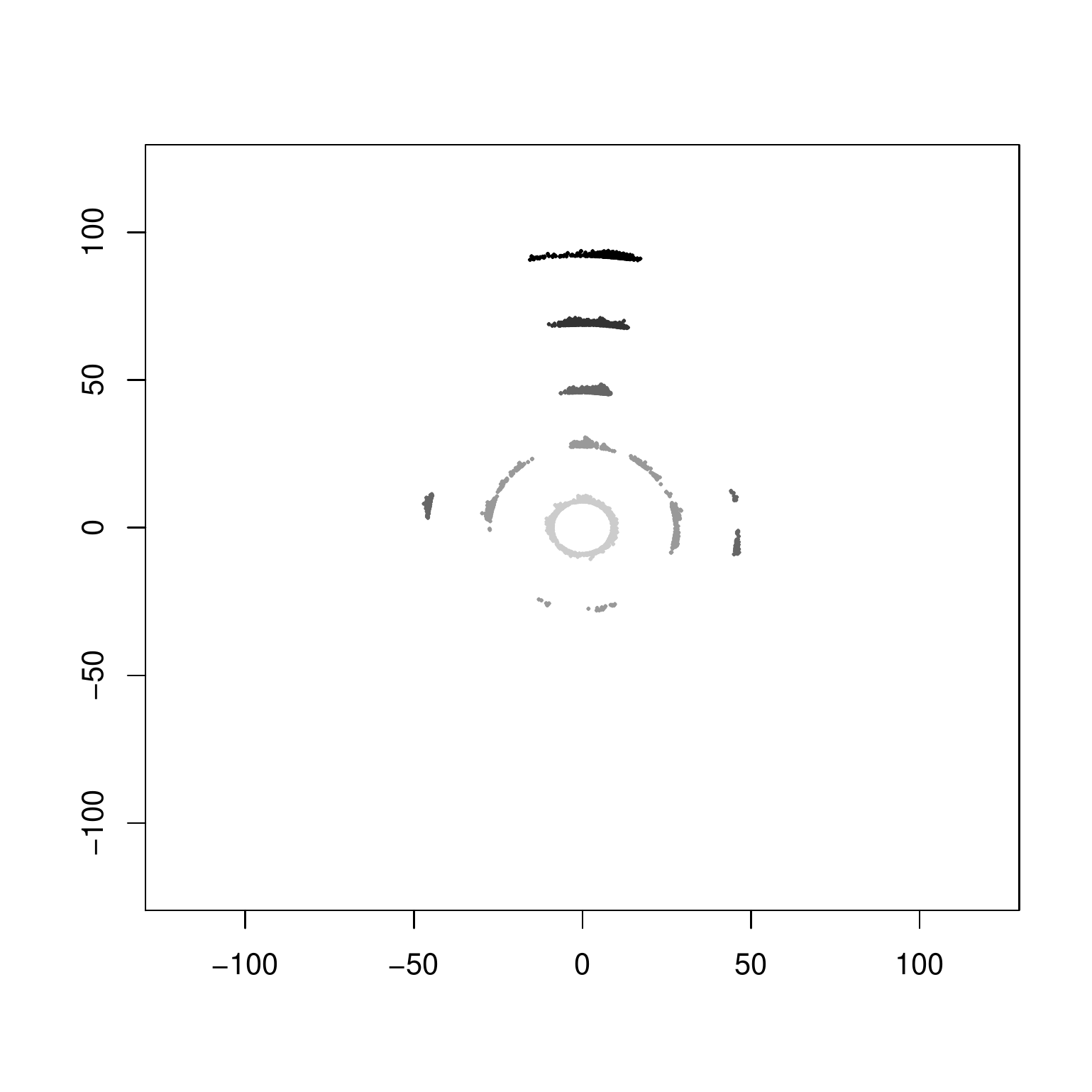} \includegraphics[width=0.495\textwidth]{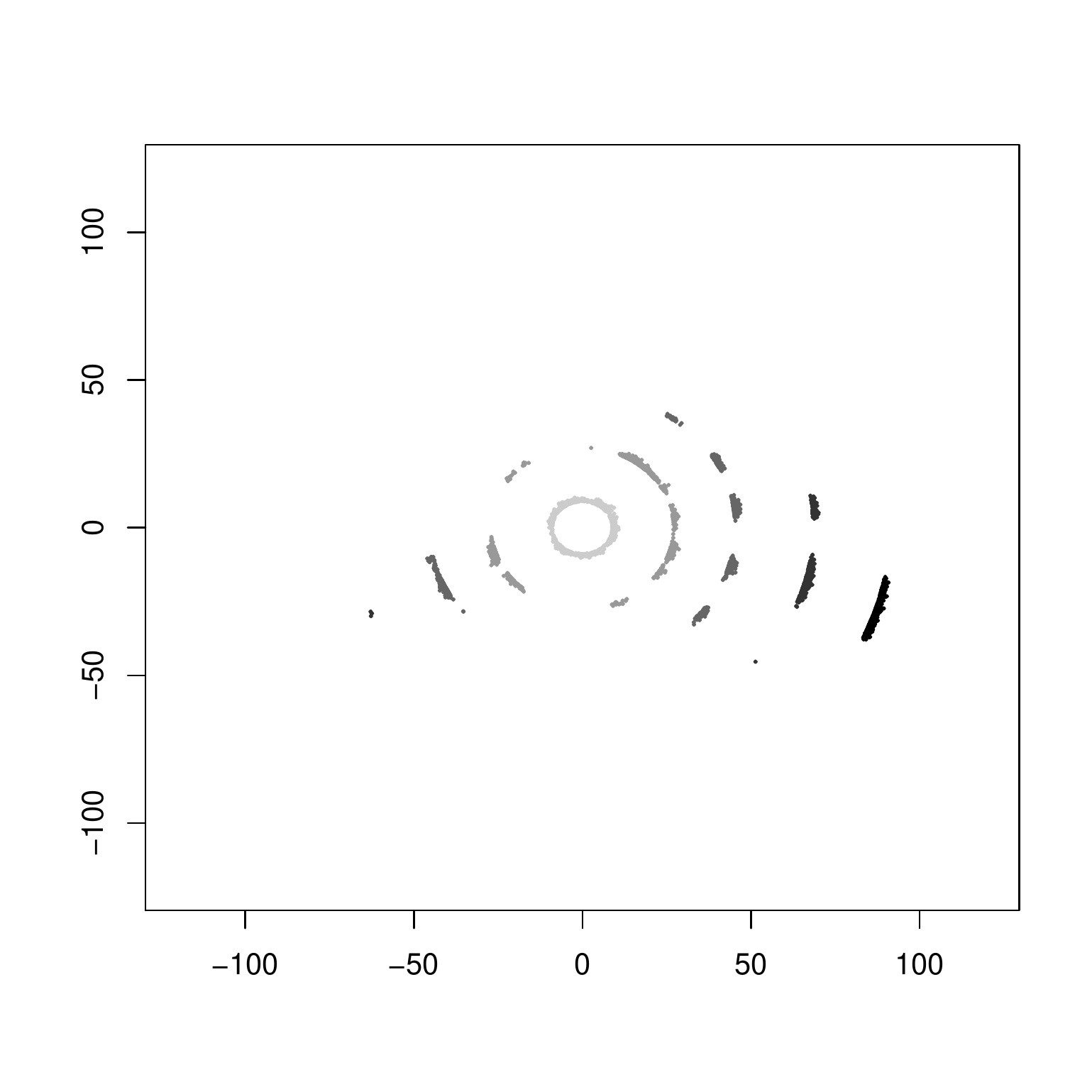} 
\caption{Two realisations of  the particle system with $N = 1000$, $d = 2$, $s(x,y) = x^2 + y^2$ and jump distribution uniform in the unit disk. The particles are plotted after $20$, $60$, $100$, $150$ and $200$ generations with decreasing brightness.}
\label{F:modulus}
\end{figure}

Simulations suggest that after an initial phase where the particles live in fragmented clusters on a circle of a given radius (which increases at linear speed), particles eventually aggregate in one clump, which travels at that speed in a random direction. A similar phenomenon is observed in simulations for the linear case (case B). Our first result makes this observation rigorous. In order to state it, it is convenient to introduce some notations. If $t > 0$ and $1 \leq n \leq N$, write $X_n(t) = R_n(t) \Theta_n(t)$, where $R_n(t) > 0$ and $\Theta_n(t) \in \bb{S}^{d-1}$ is continuous. Note that for $d \geq 2$, almost surely $X_n(t) \neq 0$ for all $t > 0$ and $1 \leq n \leq N$.

\begin{thm}\label{T:modulus} 
Let $N > 1$ and consider a Brunet--Derrida process in $\R^d$ with $N$ particles, driven by the Euclidean score function $s(x) = \|x\|$ (case A). Then,
\begin{equation}\label{E:clump_mod}
  \max_{1\leq n,m \leq N}  \frac{\| X_n(t) - X_m(t)\|}{t} \to 0,
\end{equation}
as $t \to \infty$ almost surely. 

Moreover,
\begin{equation}\label{E:dir_mod}
 \frac{R_1(t)}t \to v_N, \quad \Theta_1(t) \to \Theta,
\end{equation}
where $v_N > 0$ is a deterministic constant and $\Theta$ is 
distributed on $\bb{S}^{d-1}$. Both these convergences hold almost surely.
\end{thm}

\begin{rmk}
In the above theorem, \eqref{E:clump_mod} says that the particles eventually aggregate in one clump. On the other hand \eqref{E:dir_mod} says that the clump travels at linear speed $v_N$, in a randomly chosen direction $\Theta$. We will discuss below more precisely the diameter of the cloud of particles, which (for a fixed $N$, as $t \to \infty)$ stays of order one.
\end{rmk}

\begin{rmk}
This theorem is actually true for a more general class of Brunet--Derrida systems than the ones discussed in this introduction and, indeed, in much of the paper. See Remark \ref{R:gen} for a discussion of the class of processes to which our proofs apply.
\end{rmk}

We are also able to obtain a lower bound for the correct genealogical time for the one-dimensional process up to some mild conditions on the initial position of the particles.

A similar result holds in the linear case:

\begin{thm}\label{T:linear} 
Let $N > 1$ and consider a Brunet--Derrida process in $\R^d$ with $N$ particles, driven by the linear score function $s(x) = \< \lam, x\>$ for some $\lam \in \bb{S}^{d-1}$ (case B). Then,
\begin{equation}\label{E:clump_lin}
 \max_{1\leq n,m \leq N}  \frac{\| X_n(t) - X_m(t)\| }{t} \to 0, 
\end{equation}
as $t \to \infty$ almost surely. 

Moreover,
\begin{equation}\label{E:dir_lin}
 \frac{X_1(t)}t  \to \lam v_N,
\end{equation}
almost surely, where $v_N > 0$ is a deterministic constant.
\end{thm}

\begin{rmk}
In particular, in this case, the direction of the cloud of particles is deterministic and is simply $\lam$.
\end{rmk}

\begin{rmk}
It is not hard to see that the $v_N$ appearing in Theorem \ref{T:modulus} and \ref{T:linear} are both equal to the asymptotic speed of a one-dimensional (standard) Brunet--Derrida system. Hence, adapting a result of B\'erard and Gou\'er\'e \cite{BG} for branching Brownian motion, we get
\begin{equation}
v_N = \sqrt{2} - \frac{\pi^2}{\sqrt{2}(\log N)^2} + o((\log N)^{-2}),
\end{equation}
as $N \to \infty$.
\end{rmk}

Our next results concern the dimensions of the cloud of particles. The simulations above suggest, somewhat counterintuitively, that the cloud of particles is {more elongated} in the direction orthogonal to the fitness gradient (and the limiting direction of the cloud). This is corroborated by a close-up view of the cloud of particles (see Figure \ref{F:shape}). 

\begin{figure}
\centering
\includegraphics[scale=.4]{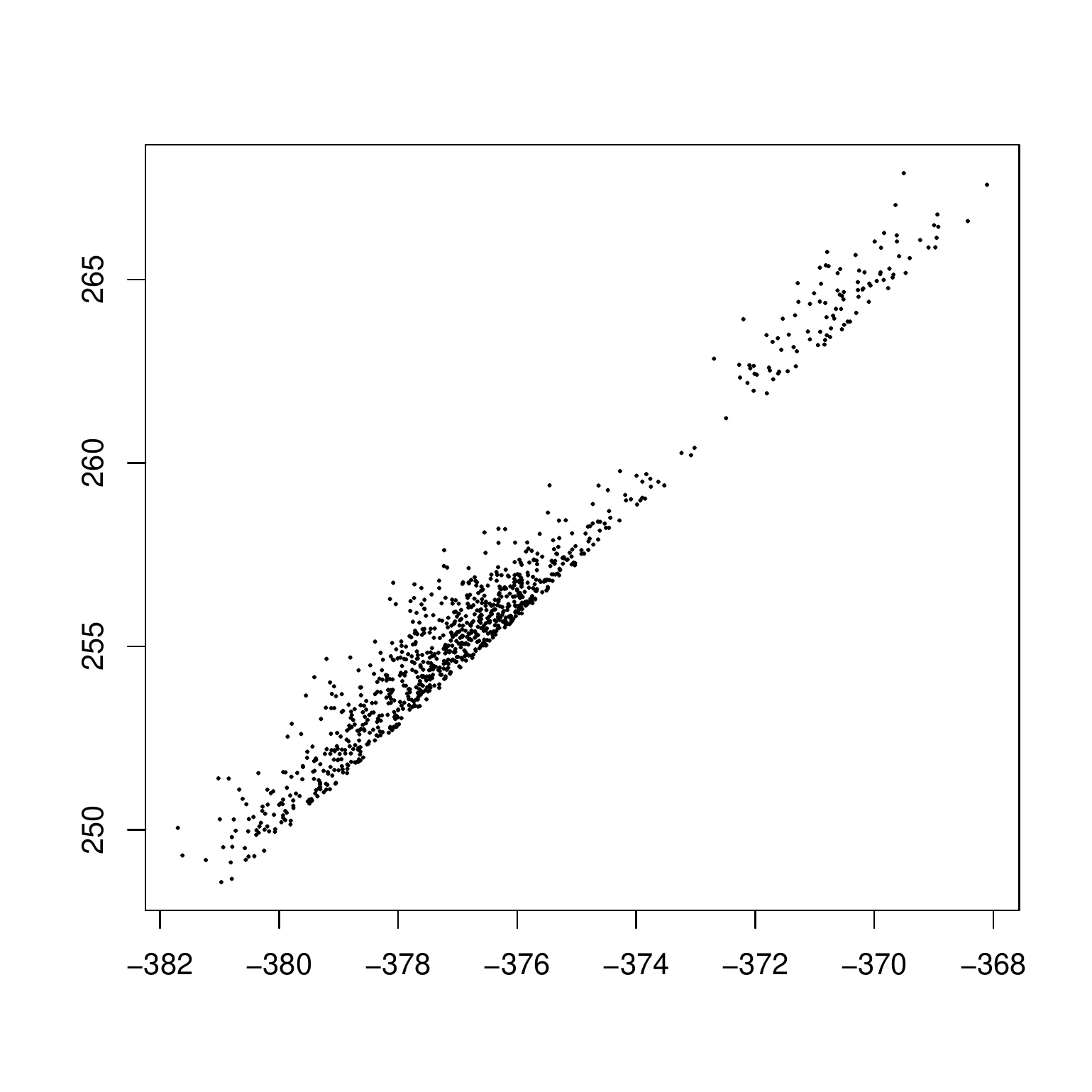}  
\includegraphics[scale=.4]{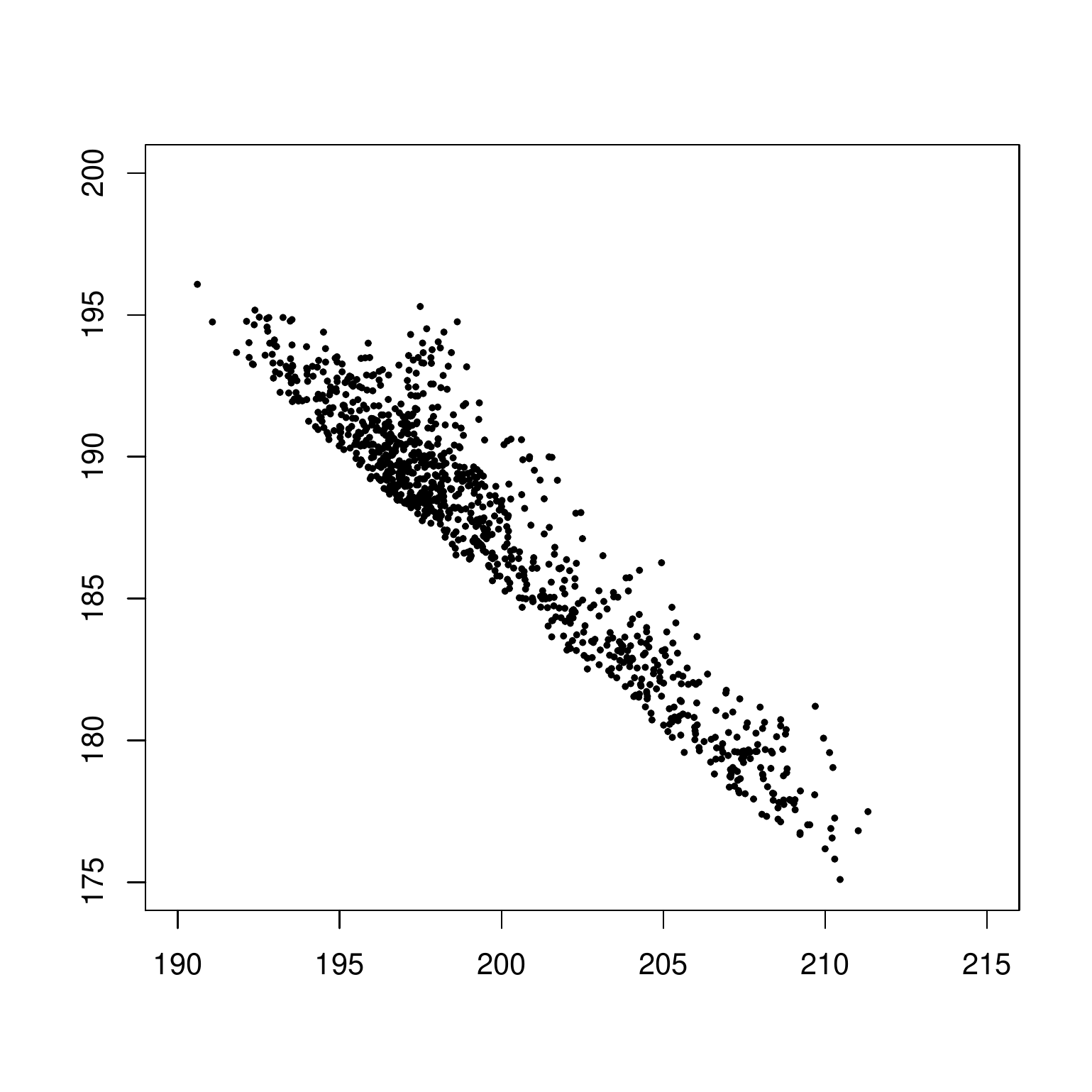} 
\caption{Close-up on the cloud of particles with $N=1000$. (a), $s(x,y) = \|x\|$, $t=1000$. (b)  $s(x,y) = x+y$, $t = 200$.}
\label{F:shape}
\end{figure}

We are able to establish this phenomenon under some reasonable assumptions on the initial condition, in case B. Fix $\lam \in \bb{S}^{d-1}$ and let $\lam^\perp$ be an arbitrary unit vector such that $\< \lam, \lam^\perp\> = 0$. Define
\[
 \diam_t = \max_{1 \leq m,n \leq N} \left| \< X_n(t) - X_m(t) , \lam \> \right| 
\]
and
\[
 \diam^\perp_t =  \max_{1 \leq m,n \leq N} \left| \< X_n(t) - X_m(t) , \lam^\perp \> \right|
\]
Fix $\lam \in \bb{S}^{d-1}$ and for all $x \in \R^d$ let $\hat x =  \< \lam, x \>$.

We introduce an assumption on the initial condition which will be used in several results below.
Let $X_1(t), \ldots, X_N(t)$ denotes the particles of a Brunet--Derrida system driven by the linear score function $s(x) = \hat x$. Let $\hat X_n(t) = \< X_n(t), \lam \>$,
and label the particles by decreasing fitness $\hat X_1(t) \geq \ldots \geq \hat X_N(t)$.
Suppose that initially the system has a particle at $X_i(0) = x$ and that for some $\delta <1$, 
\begin{equation}\label{E:init_cond}
 \sum_{n=1}^N e^{\sqrt{2} (\hat X_n(0) - \hat x )}\leq N^\delta.
\end{equation}

\begin{thm}\label{T:sausage}
Assume \eqref{E:init_cond}. Then there exists $c_\delta > 0$ (depending only on $\delta$) such that for $t = c_\delta(\log N)^3$, there exists $a > 0$ such that
\begin{equation}\label{E:shape}
 \liminf_{\eta \to 0} \liminf_{N \to \infty} \P \left( \diam_t \leq a\log N, \ \diam_t^\perp \geq \eta (\log N)^{3/2} \right) = 1.
\end{equation}
\end{thm} 


In fact it can be shown that under the same initial condition, the order of magnitude of $\diam_t $ really is $\log N$, in the sense that we also have $\diam_t \geq a' \log N$ with probability tending to 1 as $N \to \infty$, for some constant $a'<a$. The phenomenon has important consequences in population genetics which are discussed below.

\medskip We now make a series of comments on the meaning of the initial condition \eqref{E:init_cond}.

\begin{rmk}\label{R:init_cond_xi}
Intuitively, the condition \eqref{E:init_cond} says that, after projecting onto $\Span(\lam)$, only polynomially many particles lie with logarithmic distance of the maximal particle. More precisely, \eqref{E:init_cond} holds as soon as there exists $c > 0$ and $\xi < 1$ such that at most $N^\xi$ particles lie in the interval $[X_1(0) - c \log N, X_1(0)]$.
\end{rmk}

\begin{rmk}
An example of an initial condition which satisfies \eqref{E:init_cond} with high probability is as follows: sample $X_1, \ldots, X_N$ in $\R^d$ independently according to a fixed distribution such that if $\hat X = \< X, \lam\>$, then for all $x>0$,  
\begin{equation}\label{E:tail_decay}
 c_1 e^{-\alpha_1 x} \leq \P(\hat X > x) \leq c_2 e^{-\alpha_2 x}
\end{equation}
for some constant $c_1, c_2$ and $\alpha_1, \alpha_2$.
\end{rmk}

\begin{rmk}
We believe, but have been unable to prove, that if the initial condition is as in the above remark then \eqref{E:init_cond} will in fact be satisfied at arbitrary large times. Indeed, comparing with results in \cite{BBS}, we expect indeed that, at ``equilibrium" (see Section \ref{S:discussion} for definition), $X_1(0) = (1/\sqrt{2}) \log N$ and 
\[
Y_N = \sum_{n} e^{\sqrt{2} X_n(0)} \approx  N L \int_0^L e^{\sqrt{2} x} \cdot e^{-\sqrt{2} x} \sin(\frac{\pi x}L) dx  \sim c NL^2,
\]
where $L = (1/\sqrt{2}) (\log N + 3 \log \log N) $. Hence the right-hand side of \eqref{E:init_cond} should be of order $L^2$ and thus \eqref{E:init_cond} should be satisfied at equilibrium. Thus condition \eqref{E:init_cond} can be thought of as a condition specifying that the population is in a ``metastable" state, as in \cite{BBS}.
\end{rmk}

As we will see, the result in Theorem \ref{T:sausage} is closely related to estimates about the genealogical timescale (or, more precisely, the time of the most recent common ancestor) in the population. In fact, Theorem \ref{T:sausage} can be rephrased as follows:

\begin{thm}\label{T:MRCA} 
Let $N > 1$ and consider a Brunet--Derrida system with $N$ particles driven by the linear score function $s(x) = \hat x = \< x, \lam \>$. Assume that the initial condition satisfies \eqref{E:init_cond}. Then there exists $c_\delta > 0$ (depending only on $\delta$) such that any particle with fitness greater than $x$ at time $0$ has descendants alive at time $c_\delta(\log N)^3$ with probability tending to 1 as $N\to \infty$.
\end{thm} 

By projecting the particle system onto $\Span(\lam)$, we obtain a one-dimensional (standard) Brunet--Derrida system. Thus Theorem \ref{T:MRCA} applies \emph{verbatim} to such systems, which partly confirms a prediction of \cite{BDMM1, BDMM2} (see item (iii) at the start of the introduction). 

The heart of the proof relies on delicate quantitative estimates concerning the displacement of the minimal position in one-dimensional (standard) Brunet--Derrida systems. This is a difficult quantity to study rigorously, as the evolution of the minimum depends on all the particles nearby, which make up all but a negligible fraction of the population. In particular, as a process it is non-Markovian and not continuous, though in the limit $N\to \infty$ it becomes deterministic and continuous. Our result is as follows.

\begin{prop}\label{P:min_upper} Consider a (standard) one-dimensional Brunet--Derrida system with $N$ particles, ordered by decreasing fitness $X_1(t) \geq \ldots \geq X_N(t)$. Assume that the initial satisfies \eqref{E:init_cond}. Let 
\[
 \mu = \sqrt{2 - \frac{2\pi^2}{(\log N)^2}}.
\]
Then there exists $c_\delta > 0$ (depending only on $\delta$) such that as $N \to \infty$,
\begin{equation}\label{E:min_upper}
 \P \left(X_N(t) - x \leq \mu t, \, \all t \leq c_\delta(\log N)^3 \right) \to 1.
\end{equation}
\end{prop}



A corresponding lower bound for the progression of the minimal position can be established from an intermediate result of B\'erard and Gou\'er\'e \cite{BG}, with their proof adapted for branching Brownian motion. 

\begin{prop}\label{P:min_lower}
Consider a (standard) one-dimensional Brunet--Derrida system with $N$ particles, ordered by decreasing fitness $X_1(t) \geq \ldots \geq X_N(t)$. For all $\eta > 0$, there exists $c_\eta > 0$ such that for any initial condition as $N \to \infty$,
\begin{equation}\label{E:min_lower}
 \P \left( X_N(t) - X_N(0) \leq \left( \sqrt2 - \frac{(1 + \eta)\pi^2}{\sqrt2 (\log N)^2} \right) t, \, \all t \leq c_\eta (\log N)^3 \right) \to 0.
\end{equation}
\end{prop}

\subsection{Discussion and open problems}\label{S:discussion}

\emph{Long term behaviour for general fitness functions.}
Theorems \ref{T:modulus} and \ref{T:linear} establish the long-term behaviour for the cloud of particles for the two special cases where the function $s$ is either the Euclidean norm or a linear function. In both cases, the cloud escapes to $\infty$ at positive speed in a possibly random direction. It would be interesting to see how general a phenomenon this is. For instance, assume that $s:\R^d\to \R$ is a smooth, unbounded convex function. What can be said about the long-term behaviour then? One first observation is that the cloud of particles should essentially stay concentrated on level sets of the function $s$.

\medskip \emph{Genealogy.}
In both cases studied here (Euclidean case or case A, and linear case or case B), we observe that the population lines up on an essentially one-dimensional subspace of $\R^d$. For truly one-dimensional systems, it is predicted that the Bolthausen--Sznitman coalescent describes the genealogy of a sample from the population, after rescaling time by $(\log N)^3$. It is therefore reasonable to predict that the same property will hold in higher dimensions as well, at least in cases A and B and perhaps more generally as well, suggesting that the Bolthausen--Sznitman coalescent is a universal scaling limit in all dimensions, subject to assumptions on the function $s$.

\medskip \emph{Equilibrium shape in one dimension.}
Consider the empirical distribution of a (standard) one-dimensional Brunet--Derrida particle system.
\[
 \nu^N_t = \frac1N \sum_{n = 1}^N \delta_{X^N_n(t)},
\]
and the associated c\`adl\`ag empirical tail distribution 
\[
 F^N(t,x) = \frac1N \sum_{n=1}^N \bo 1 \{X^N_n(t) > x \} = \nu^N_t((x, \infty)).
\]
It is not hard to see that the system of particles, viewed from the minimum position at time $t$, has regeneration times and therefore $F^N(t, x+ X_N(t))$ converges pointwise to some limit distribution $F^N_{eq}(x)$ as $t \to \infty$, wherever $F^N_{eq}$ is continuous. It is natural to start the particle system in some initial condition distributed according to $F^N_{eq}$ and ask for its properties. We believe, but have been unable to prove, that $F^N_{eq}$ satisfies \eqref{E:init_cond}. In fact, we make the following conjecture about $F_{eq}^N$.

Reasoning by analogy with the results of Durrett and Remenik \cite{DR}, and using the martingale problem for the empirical distributions of a free branching Brownian motion (see for example, Lemma 1.10 in Etheridge \cite{Etheridge}), we expect $F^N_{eq}(t,x)$ to converge in distribution to $F(t,x)$, the solution to the free boundary problem:
\begin{equation}\label{E:FB}
\begin{cases} 
 & \displaystyle\frac{\partial F}{\partial t} = \frac12\frac{\partial^2 F}{\partial x^2} + F(t,x) \ \ \all x > \gamma(t), \\
 & F(t,x) = 1 \ \ \all x \leq \gamma(t),
 \end{cases}
\end{equation}
where $\gamma : [0, \infty) \to \bb R$ is a continuous, increasing function starting from $0$, which is part of the unknown in \eqref{E:FB}. (Note that Durrett and Remenik's argument breaks down for particles that perform Brownian motion, as it is essential in their coupling that particles sit still in between branching events. It is unclear how to adapt their argument to the case of Brownian motion). The first equation is simply the linearised FKPP equation \eqref{E:FKPP}, which is satisfied asymptotically as $x \to \infty$ by the distribution tail of the position of the rightmost particle in a (free) branching Brownian motion. The second equation on the other hand represents the effect of selection, and $\gamma(t)$ then describes the limiting position of the minimal particle.
\cite{DR} shows the existence of a family of travelling wave solutions for a class of problems similar to \eqref{E:FB}. Here the traveling wave solutions can be found explicitly: if $F(t,x) = W(x-ct)$ solves \eqref{E:FB}, we find
\[
 -cW' = \frac12W'' + W.
\]
This is a second order differential equation which, as is well known, has positive solutions only if the speed $c$ of the traveling wave satisfies $c\geq \sqrt{2}$. For $c=\sqrt{2}$ the solution is 
\begin{equation}
 W_{*}(x) = (\sqrt2x + 1)e^{-\sqrt2x}.
\end{equation}
Turning back to $F^N_{eq}$, stationarity suggests that $F^N_{eq}$ is in the limit as $N\to \infty$ a traveling wave solution of \eqref{E:FB}. But by Proposition \ref{P:min_upper} if $F_{eq}^N$ is a travelling wave solution the speed would have to be at most $\sqrt{2}$, and so equal to $\sqrt{2}$. Therefore, we conjecture that
\begin{equation}\label{E:conj_limit_shape}
 F^N_{eq}(x) \to W_*(x)
\end{equation}
uniformly on compact sets as $N\to \infty$. 

\medskip \emph{Equilibrium shape in high dimensions.} Let $d\geq 1$ and fix an arbitrary smooth selection function $s$. For reasons similar to above, it is possible to define a notion of limiting equilibrium shape of the system as $t\to \infty$. Theorem \ref{T:sausage} gives information about the dimensions (width and length) of the limiting shape in case B. However, an inspection of the simulations suggests that particles are far from uniformly distributed within that shape. In the direction $\lam$, we expect the density of particles to be close to $W_*(x)$ for the same reasons as above. In the transverse direction $\lam^\perp$ however, particles appear somewhat `clustered'. Indeed, this is to be expected given the hierarchical structure of the Bolthausen--Sznitman coalescent. Clusters of particles represent groups of particles coming from a close common ancestor. However,  clusters are also intertwined because of heat kernel smoothing. It is an interesting question to identify the density of particles at equilibrium.

\subsection{Biological applications: the effect of recombination}

As alluded to in earlier parts of this introduction, our Brunet--Derrida system in more than one dimension can be thought of as a model for the effect of selection on multiple linked loci. In this interpretation, we track the fitness of not one but $d$ loci in a population of size $N$. Each particle corresponds to one-half of an individual's genetic material, and each of the $d$ coordinates of that particle represents the fitness at the corresponding locus. Her total fitness will then be a function of these $d$ values, typically just the sum. In this interpretation, we are assuming that the total fitness of each particle evolves like independent Brownian motions and branch independently of one another, which is a simplification because in reality, two particles -- making up one individual -- will branch simultaneously. For the same reasons, whereas in our model we only remove one particle at a time, it would be make more sense to remove two particles at once (also making up an individual). But we choose to ignore the correlations between an individual's two genetic halves, and still believe that the model captures some essential features of reproduction. Note that, as specified above, the model ignores the possibility of recombination. But we will precisely explain the effect of adding recombination to the model in a moment and show that it leads to an increase in overall fitness. 

It has been a longstanding problem in evolutionary biology to explain the ubiquitous nature of diploid populations over haploid populations. Indeed, in diploid populations the chance of a particular gene being transmitted to an offspring is only 50\%, whereas it is 100\% in haploid populations! This would suggest that haploid populations are far more advantageous from the point of view of a particular gene. This paradox was in fact raised soon after the introduction of Darwin's theory of natural selection and evolution. 

As early as 1889, Weismann \cite{Weismann} advocated that sex functions to provide variation for natural selection to act upon. However it is fair to say that no real consensus was achieved in the population genetics community, especially after influential arguments by Williams \cite{Williams} raised doubts on Weismann's theory. The controversy reached the point where understanding the advantage of sexual reproduction became the ``queen of problems in evolutionary biology" \cite{Bell}. We refer to Burt \cite{Burt} for an excellent and highly readable survey of this question. 

In his study of the problem, Burt \cite{Burt} observed that his models led to a negative correlation between the fitness on the two chromosomes, which is equivalent to a cloud of particles being spread out in the direction orthogonal to the fitness gradient (see Fig. 1D of \cite{Burt}). He then reasoned that a small amount of recombination would lead to a reduction in this correlation and greater variance in the overall fitness, ultimately leading to a fitter population, as can be seen on Figure \ref{F:recomb}. Thus Theorem \ref{T:sausage} can be viewed as a rigorous justification of the Weissmanian proposal in this setting. 

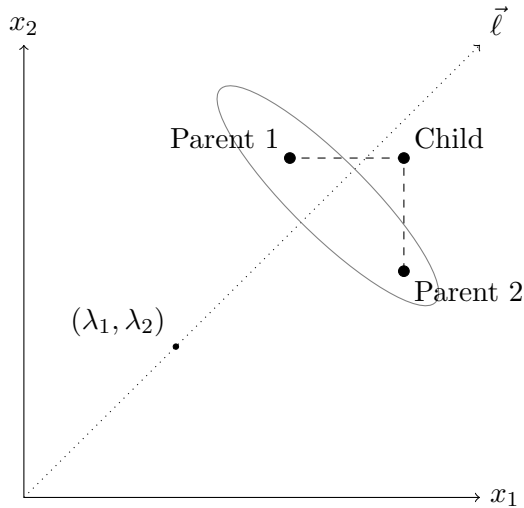
\begin{figure}
\centering
\begin{tikzpicture}
 \draw [<->] (6,0) -- (0,0) -- (0,6);
 \draw [->, dotted] (0,0) -- (6,6);
 \node [above] at (0,6) {$x_2$};
 \node [right] at (6,0) {$x_1$};
 \node [above right] at (6,6) {$\vec \ell$};
 \draw [fill] (2,2) circle (1pt) node [above left] {$(\lam_1, \lam_2)$};
 \draw [rotate around={45:(4,4)}, gray] (4,4) ellipse (0.5 and 2);
 \draw [fill] (3.5,4.5) circle (2pt) node [above left] {Parent 1};
 \draw [fill] (5,3) circle (2pt) node [below right] {Parent 2};
 \draw [dashed] (5,3) -- (5,4.5) -- (3.5,4.5);
 \draw [fill] (5,4.5) circle (2pt) node [above right] {Child};
\end{tikzpicture} \\
\caption{In the presence of recombination, the offspring of two individuals with positions $(x_1, y_1)$ and $(x_2, y_2)$ is either $(x_1, y_2)$ or $(x_2, y_1)$. If the fitness across loci is negatively correlated (i.e., if the shape of the cloud is elongated in the direction transverse to fitness gradient), this leads to an overall increase in the variance of the fitness distribution, even though the mean in unchanged. In turn, this results in increased response to natural selection.}
\label{F:recomb}
\end{figure}

\medskip \textbf{Acknowledgements.} We are grateful to Julien Berestycki for a number of fruitful conversations at several stages of this project. In particular we learned of the potential relation between multidimensional Brunet--Derrida systems and the role of recombination from him, and we are grateful to have been shown a draft of \cite{BerestyckiYu} which raised that issue. 


\section{Proof of Theorems \ref{T:modulus} and \ref{T:linear}}

We are now ready to give a proof of Theorem \ref{T:modulus} and \ref{T:linear}. In this section $N > 1$ is fixed.

We first begin with a formal construction of the Brunet--Derrida particle system. Let $(J_i)_{i\geq 0}$ be the jump times of a Poisson process with rate $N$ with $J_0 = 0$, and let $(K_i)_{i\geq 1}$ be an independent sequence of i.i.d. uniform random variables on $\{1, \ldots, N\}$ . The process is started in some given initial condition. Then inductively, for each $i \geq 1$, assuming that the system is defined up to time $J_{i-1}$ with $s(X_1(J_{i - 1})) \geq \ldots \geq s(X_N(J_{i - 1}))$, we define
\begin{equation}\label{E:coup_cts}
 X_n(t) = X_n(J_{i - 1}) + Z_n(t - J_{i - 1}), \quad t \in [J_{i-1}, J_i),
\end{equation}
where $(Z_n(t), t \geq 0)$ are independent Brownian motions in $\R^d$, independent from $(K_i)$ and $(J_i)$. At time $J_i$, we duplicate particle $X_{K_i}(J_i^-)$ and remove the particle $\min_{1\leq n \leq N} s(X_n(J_i^-))$. Note that if the duplicated particle is the particle of minimal score, the net effect is that nothing happens. We now relabel the particles over this interval in the usual convention of descending fitness so
\[
 s(X_1(t)) \geq \ldots \geq s(X_N(t)), \quad t \in [J_{i-1}, J_i].
\]

\subsection{Proof of Theorem \ref{T:linear}}

We start with a few elementary facts about (free) branching Brownian motion $\bar X_1(t), \ldots, \bar X_{\bar N(t)}(t)$ in $\R$, where $\bar N(t)$ is the number of particles at time $t$. In keeping with our convention for this article we order particles from right to left. We assume that initially there is one particle at the origin.

The following lemma is a trivial but useful result to relate the statistics for all the particles alive in a free branching Brownian motion to a single Brownian motion and is sometimes known in the literature as the many-to-one lemma (see for example \cite{HH}).

\begin{lem}\label{L:many_to_one} 
Let $T$ be a random stopping time of the filtration $\bar \F_t = \sigma(\bar X_i(s), i\leq \bar N(t), s\leq t)$, and assume that $T$ is almost surely finite. For $s < T$ and each $i \leq \bar N(T)$, let $\bar Y_i(s)$ be the position of the unique ancestor of $\bar X_i(T)$. Then for any bounded measurable functional $g$ on the path space $C([0, \infty))$,
\[
 \Ex \left[ \sum_{i \leq \bar N(T)} g((\bar Y_i(s))_{s \leq T}) \right] = \Ex[e^T g((B_s)_{s \leq T})],
\]
where $(B_s)_{s \geq 0}$ is a standard Brownian motion. 
\end{lem} 

With the many-to-one lemma, we can obtain a naive bound for the maximum displacement of a particle at time $t$ from its parent at time $0$, as well as the running maximum.

\begin{lem}\label{L:maximum} 
For any $K > 0$,
\[
 \P(\bar X_1(t) \geq \sqrt 2 t + K) \leq e^{-\sqrt2 K}.
\]
Moreover,
\[
 \P(\sup_{s \leq t} \bar X_1(s) \geq \sqrt 2 t + K) \leq 2e^{-\sqrt2 K}.
\]
\end{lem}

\begin{proof} 
By Lemma \ref{L:many_to_one},
\begin{align*}
 \P(\bar X_1(t) \geq \sqrt 2t + K) & \leq \Ex \left[ \sum_{i \leq \bar N(t)} \bo 1 \{\bar X_i(t) \geq \sqrt 2 t + K \} \right] \\
 & = e^t \P(B_t \geq \sqrt 2t + K)\\
 & \leq e^t e^{-\frac1{2t}(\sqrt 2t + K)^2} \leq e^{-\sqrt2 K},
\end{align*}
where we use the well known tail bound for a standard normal random variable $X$, and $a > 0$, $\P(X \geq a) \leq e^{-\frac{a^2}2}$.

For the historic maximum, a similar argument shows
\[
 \P(\sup_{s \leq t}  \bar X_1(s) \geq \sqrt 2t + K) \leq e^t \P(\sup_{s \leq t} B_s \geq \sqrt 2t + K).
\]
Using the reflection principle,
\[
 \P(\sup_{s \leq t} B_s \geq \sqrt 2t + K) = 2 \P(B_t \geq \sqrt 2t + K),
\]
and the result follows.
\end{proof}

When a Brunet--Derrida system is driven by a score function $s(x) = \< x, \lam \>$, where $\lam \in \bb S^{d-1}$, we have already noted that after projecting the particle system onto $\Span(\lam)$, we recover a standard one-dimensional Brunet--Derrida system. For such systems, we have an easy but useful coupling used by B\'erard and Gou\'er\'e \cite{BG}.

\begin{lem}\label{L:monotone} 
Consider two (standard) one-dimensional Brunet--Derrida systems, $(X_n(t), 1\leq n \leq N)_{t\geq 0}$ and $(Y_n(t), 1\leq n \leq N')_{t\geq 0}$, $N \leq N'$, which are initially ordered $X(0) \prec Y(0)$ in the sense of stochastic domination: that is, there is a coupling of $X(0)$ and $Y(0)$ such that
\[
 Y_1(0) \geq X_1(0); \ldots ; Y_N(0) \geq X_N(0).
\]
Then we can couple $X(t)$ and $Y(t)$ for all time such that $X(t) \prec Y(t)$ for all time $t \geq 0$.
\end{lem} 

\begin{proof}
Construct $X(t)$ and $Y(t)$ using \eqref{E:coup_cts} with the same jump times $(J_i)_{i \geq 0}$ and the same family $(Z_n(t), t \geq 0)$ of independent Brownian motions in $\R$.
\end{proof}

Adapting the (easy) proof of Proposition 2 of \cite{BG} one obtains:

\begin{lem}\label{L:1D_speed} 
Consider a one-dimensional Brunet--Derrida system initially with $X_1(0) \geq \ldots \geq X_N(0) = 0$. Then
\[
 \frac{X_1(t)}t \to v_N,
\]
almost surely, where $v_N > 0$ is a deterministic constant.
\end{lem}

The argument is based on the monotonicity of Lemma \ref{L:monotone} and Kingman's sub additive ergodic theorem. To see that $v_N > 0$ for $N > 1$, we observe that $v_1 = 0$ and that there is a straightforward strengthening of Proposition 3 of \cite{BG} to see that $(v_N, N \geq 1)$ is strictly increasing.

The same argument also applies to $X_N(t)$, but \emph{a priori} the limiting velocity $v'_N$ might be distinct from $v_N$. In fact the following lemma, which can be proved in the same fashion as Proposition 1 of of \cite{BG}, shows that $v_N = v'_N$.

\begin{lem}\label{L:diameter}
Let $(X_n(s), 1 \leq n \leq N)_{s \geq 0}$ be a (standard) one-dimensional Brunet--Derrida system. Then for all $\e > 0$ and $t > (1 + \kappa) \log N$ for some $\kappa > 0$,
\[
 \lim_{N \to \infty} \P \left( X_1(t) - X_N(t) \geq (3\sqrt2 + \e)\log N \right) = 0.
\]
\end{lem}

\begin{cor}\label{C:diameter}
For all $N > 1$ and $\e > 0$,
\[
 \lim_{s \to \infty} \P \left( \frac{X_1(s) - X_N(s)}s \geq \e \right) = 0.
\]
\end{cor}

Going back to a Brunet--Derrida system in $\R^d$, let $H = \{ x \in \R^d: \< x, \lam \> = 0 \}$ be the orthogonal hyperplane to $\lam$ and let $p_H$ be the orthogonal projection onto $H$.

Referring back to the construction of the system via \eqref{E:coup_cts}, conditional on $\cF_{J_{i-1}}$ (where $\cF_t$ is the filtration generated by the whole system up to time $t$), particles perform $(d-1)$-dimensional Brownian motion on $H$ independent of the motion in $\Span(\lam)$ up to time $J_i$ for every $i \geq 1$. Moreover, since $s(x) = \< x , \lam \>$, $p_H(X_m(J_i))$ is independent of the event that the particle $X_m$ survives a branching event at time $J_i$. Together, these two properties imply by induction that the path of a particle conditioned to survive until time $t$ when projected onto $H$ has the law of a standard $(d-1)$-dimensional Brownian motion. In other words, if $X_n(t)$ is a surviving particle at time $t$ and $Y_n(s)$ is the ancestor of $X_n(t)$ at time $s \leq t$, then $(p_H(Y_n(s)), s \leq t)$ is a standard $(d-1)$-dimensional Brownian motion. Therefore, for all $1 \leq n \leq N$,
\[
 \frac{\|X_n(t)\|_H}{t} = \frac{\|Y_n(t)\|_H}{t} = \frac{\|p_H(Y_n(t))\|}{t} \to 0,
\]
almost surely for all $1 \leq n \leq N$. Therefore
\[
 \max_{1\leq n,m \leq N}  \frac{\| X_n(t) - X_m(t)\|_H}{t} \leq 2 \max_{1 \leq n \leq N} \frac{\|X_n(t)\|_H}{t} \to 0.
\]
Together with Lemma \ref{L:1D_speed}, this completes the proof of Theorem \ref{T:linear}.
\qed

\subsection{Proof of Theorem \ref{T:modulus}}

Assume now $s(x) = ||x||$. Recall in this setting, for $X_n(t) \neq 0$, we write $X_n(t) = R_n(t)\Theta_n(t)$ where $R_n(t) > 0$ and $\Theta_n(t) \in \bb S^{d-1}$ is continuous whenever $X_n(t)$ is continuous. Note also for $d \geq 2$, $d$-dimensional Brownian motion almost surely never hits $0$. Hence, except for any particles initially at $0$, this decomposition is always well-defined. We can work around particles starting from $0$ by instead taking the system at time $t > 0$ as its initial state without altering the proofs. Therefore, without loss of generality, we shall assume from here on that $R_N(0) > 0$ and we need not worry about any particles at $0$.

When considering $(R_n(t), 1 \leq n \leq N)$, we can work in a one-dimensional setting and construct the system in a similar manner as before, except now the displacement step \eqref{E:coup_cts} becomes
\begin{equation}\label{E:coup_cts_mod}
 R_n(t) = R_n(J_{i - 1}) + S_n(R_n(i - 1), t - J_{i - 1}), \quad t \in [J_{i-1}, J_i),
\end{equation}
where $(S_n(r,t), 1 \leq n \leq N)$ are an independent family of solutions to the Bessel stochastic differential equation
\begin{equation}\label{E:bessel}
dS(r, t) = dB(t) + \frac{d - 1}{2S(r, t)} \, dt,
\end{equation}
where $B(t)$ is a standard one-dimensional Brownian motion, and $S(r,t)$ is a solution starting from $S(r,0) = r$. In this construction, we see immediately that $S(r,t)$ is stochastically decreasing in $r$ and that as $r \to \infty$, $S(r,t)$ converges to a standard one-dimensional Brownian motion.

Recall that the default ordering is in descending fitness, and so $R_1(t) \geq \ldots \geq R_N(t)$.

\begin{lem}\label{L:modulus_R}
For all $N > 1$,
\[
 \frac{R_1(t) - R_N(t)}{t} \to 0,
\]
as $t \to \infty$ almost surely. Moreover,
\[
 \frac {R_1(t)}t \to v_N,
\]
almost surely, where $v_N > 0$ is a deterministic constant.
\end{lem}

\begin{proof}
Given $(R_n(t), 1 \leq n \leq N)$ constructed in the usual manner and with \eqref{E:coup_cts_mod}, we define a family of one-dimensional Brunet--Derrida systems $(Y^\e_n(t), 1 \leq n \leq N)$ constructed in the same manner as $(R_n(t), 1 \leq n \leq N)$ with with the same $(J_i)$, $(K_i)$, but with the displacement step
\begin{equation}\label{E:coup_cts_drift}
 Y^\e_n(t) = Y^\e_n(J_{i - 1}) + W^\e_n(t - J_{i - 1}), \quad t \in [J_{i-1}, J_i),
\end{equation}
where $(W^\e_n(t), 1 \leq n \leq N)$ are independent Brownian motions in $\R$ with $\e$ drift. These processes satisfy the stochastic differential equation
\begin{equation}\label{E:BMdrift}
dW^\e(t) = dB(t) + \e \, dt.
\end{equation}
Suppose we couple the family $(Y^\e_n(t), 1 \leq n \leq N)$ to $(R_n(t), 1 \leq n \leq N)$ by using the same underlying $B(t)$ to drive the solutions to \eqref{E:bessel} and \eqref{E:BMdrift} for each $i$ and $n$. Then we see that under this coupling
\[
 \liminf_{t \to \infty} \frac{R_N(t)}t \geq \liminf_{t \to \infty} \frac{Y^0_N(t)}t.
\]
But $(Y^0_n(t), 1 \leq n \leq N)$ is a standard one-dimensional Brunet--Derrida system and by Lemma \ref{L:1D_speed},
\[
 \lim_{t \to \infty} \frac{Y^0_N(t)}t = v_N
\]
almost surely for some deterministic constant $v_N > 0$. Therefore, almost surely, $R_N(t) \to \infty$. So under this coupling, for every $\e > 0$,
\[
 \limsup_{t \to \infty} \frac{R_1(t)}t \leq \limsup_{t \to \infty} \frac{Y^\e_1(t)}t
\]
almost surely. However, we note that if the drift is a constant equal to $\eps$, then $Y_i^\eps(t) = Y_i^0(t) + \eps t$ for all $1\le i \le N$ and all $t \ge 0$, hence by Lemma \ref{L:1D_speed},
\[
 \lim_{t \to \infty} \frac{Y^\e_1(t)}t = v_N + \e,
\]
almost surely. Since $\e > 0$ was arbitrary, we have that almost surely
\[
 \lim_{t \to \infty} \frac{R_1(t)}t = \lim_{t \to \infty} \frac{R_N(t)}t = v_N.
\]
\end{proof}

Having established the asymptotic behaviour of for $R(t)$, we now turn our attention to $\Theta(t)$. The main idea here is that the time to the most recent common ancestor for all $N$ particles, can be naively dominated uniformly over all time. We shall formalise this statement with the following lemma.

\begin{lem}\label{L:geo_MRCA_bound}
Let $\tau(t)$ be the time to the most recent common ancestor for $X_1(t), \ldots, X_N(t)$. Then for all $t$ sufficiently large, $\tau(t) - 1$ is stochastically dominated by a geometric random variable of parameter $p$, where $p > 0$.
\end{lem}

\begin{proof}
Let $s \geq 0$ be an integer, and consider the system $X_1(s), \ldots, X_N(s)$ at time $s$. We assume that the Brunet--Derrida system $(X_n(t), 1\le n \le N)_{t \ge s}$ is obtained from a free branching Brownian motion $(\bar X_n(t), 1\le n \le \bar N(t))_{t \ge 0}$ in the obvious manner. 
Let $A_s$ be the event that, for this free process, when the particle located at $X_1(s)$ at time $s$ first branches after time $s>0$, its score is $\geq R_1(s) + 1$, and it subsequently produces at least $N$ offspring by time $s + 1$ whose score always stays above $R_1(s) +1/2$. Let $B_s$ be the event that for the free process, the particles initially located at $X_2(s), \ldots, X_N(s)$ do not branch before time $s + 1$ and that
\begin{equation}\label{E:regen_event}
 \sup_{2 \leq n \leq N} \sup_{t \in [s,s+1]}  \| Y_n(t) - Y_n(s)\|  \leq 1/2,
\end{equation}
where $Y_n(t)$ is the location at time $t$ of the descendant of the particle located at $X_n(s) $ at time $s$. Note that $Y_n(s)$ is well-defined since $X_n(s)$ has a unique descendant for $2 \leq n \leq N$.

Note that $A_s$ and $B_s$ are independent events. Moreover, $B_s$ is independent of $X(s)$, so there exists $p_2>0$ such that
\[
 \P(B_s | \cF_s ) = \P(B_s) \ge p_2
\]
almost surely for all $s$, where $\cF_s$ denotes the filtration generated by the entire process up to time $s$. Likewise, $A_s$ given $R_1(s)$ is independent of $\cF_s$. To lose the dependence on $R_1(s)$, we use an analogous coupling as in the proof of Lemma \ref{L:modulus_R} where we stochastically bound $(R_n(t), 1 \leq n \leq N)$ from below by a standard one-dimensional Brunet--Derrida $(Y_n^0(t), 1 \leq n \leq N)$. We define the event $A'_s$ to the event that, for a \emph{one-dimensional} free branching Brownian motion, a particle located at $R_1(s)$ first branches after time $s > 0$, its score is $\geq R_1(s) + 1$, and it subsequently produces at least $N$ offspring by time $s + 1$ whose score always stays above $R_1(s) + 1/2$. We now have that $A'_s$ is independent of $R_1(s)$ and therefore $\cF_s$ and
\[
 \P(A_s |\cF_s) \geq \P(A'_s | \cF_s) = \P(A'_s).
\]
So there exists $p_1>0$ such that 
\[
 \P(A_s |\cF_s) \ge p_1
\]
almost surely for all $s$. We call $p = p_1 p_2>0$, and deduce from the above that if $G_s = A_s \cap B_s$, 
\[
 \P(G_s |\cF_s) \ge p
\] 
almost surely for all $s$. Note that when $A_s \cap B_s$ occurs, all the particles at time $s+1$ in the Brunet--Derrida system necessarily descend from the maximum particle at time $s$. Hence $\tau(s + 1) \leq 1$. 

Applying this argument iteratively, we deduce that
\[
 \P(\tau (t) > k) \leq \P(G_{t-k}^\complement \cap G_{t-k+1}^\complement,\cap \ldots \cap G_{t-1}^\complement)
\]
from which the result follows.
\end{proof}

With Lemma \ref{L:geo_MRCA_bound}, we are now in a position to complete the proof of \eqref{E:clump_mod} with the following lemma. Endow $\bb S^{d-1}$ with the usual spherical metric $D$: for $\Theta_1, \Theta_2 \in \bb S^{d-1}$, let $D(\Theta_1, \Theta_2)$ be the distance on the sphere. In $\R^d$,
\[
 D(\Theta_1, \Theta_2) = \cos^{-1} \< \Theta_1 , \Theta_2 \>.
\]

\begin{lem}\label{L:theta_clump}
For all $N > 1$,
\[
 \max_{1 \leq m,n \leq N} D(\Theta_m(t), \Theta_n(t)) \to 0
\]
as $t \to \infty$ almost surely.
\end{lem}

\begin{proof}
Given two particles $X_m(s), X_n(t) \in \R^d$, let $r =  \|X_m(s) - X_n(t)\|$ and assume for now that $r\le R_n(t) = \|X_n(t)\|$.
Then a simple geometric argument (see Figure \ref{F:trig}) shows that the distance $D(\Theta_m(s), \Theta_n(t))$ is biggest if $X_m(s)$ is perpendicular to $X_m(s) - X_n(t)$. Hence for $r \leq R_n(t)$,
\begin{equation}\label{E:trig}
 D(\Theta_m(s), \Theta_n(t)) \leq \sin^{-1} \left( \frac r{R_n(t)} \right) \leq \frac{\pi r}{2 R_n(t)},
\end{equation}
since $\sin^{-1} (x) \leq \frac\pi2 x$ for all $0 \leq x \leq 1$.
\begin{figure}
\centering 
\begin{tikzpicture}
 \draw [fill] (-3.6,1.5) circle [radius=2pt] node [above] {$0$};
 \draw [fill] (0,0) circle [radius=2pt] node [below] {$X_n(t)$};
 \draw [dashed] (0,0) circle [radius=1.5];
 \draw [-] (0,0) -- (0,1.5) node [right] at (0,0.75) {$r$};
 \draw [fill] (0,1.5) circle [radius=2pt] node [above] {$X_m(s)$};
 \draw [-] (0,0) -- (-3.6,1.5) node [below left] at (-1.8,0.75) {$R_n(t)$};
 \draw [-] (0,1.5) -- (-3.6,1.5);
 \draw [-] (0,1.3) -- (-0.2,1.3) -- (-0.2,1.5);
 \draw [-] (-3,1.5) arc [radius=0.6, start angle=0, end angle=-22.6];
 \draw [-] (-3.3,1.45) -- (-3.6,1.15) node [below left] {$D(\Theta_m(s),\Theta_n(t))$};
\end{tikzpicture} \\
\caption{Proof of \eqref{E:trig}. The angle is maximised when the triangle formed by $0$, $X_m(s)$ and $X_n(t)$ is rectilinear.}
\label{F:trig}
\end{figure}
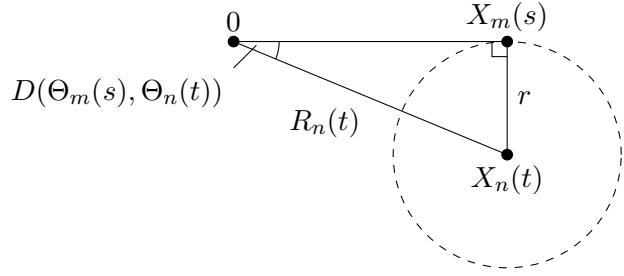

Given $0 < \Delta < t$ and let $\tau = \tau(t)$ be the time to the most recent common ancestor of all the surviving particles at time $t$. We define $\tau = \infty$ should there be no such ancestor. We first note that
\[
 \P(\tau \geq \Delta ) \leq (1- p)^\Delta
\]
where $p$ is as in Lemma \ref{L:geo_MRCA_bound}. Hence picking $\Delta = C_1 \log t$ for some sufficiently large $C_1 > 0$, and applying the first Borel--Cantelli lemma shows that there exists $T_1 > 0$, possibly random, such that almost surely, $\tau \leq \Delta$ for all $t > T_1$.

On the event $\{\tau \leq \Delta\}$, let $X_k(t - \tau)$ be the position of the most recent common ancestor of all the surviving particles at time $t$. Since $\sup D \leq \pi$, using \eqref{E:trig}, we have:
\begin{align}
  D(\Theta_m(t),\Theta_n(t)) & \leq D(\Theta_n(t),\Theta_k(t - \tau)) + D(\Theta_m(t),\Theta_k(t - \tau)) \notag \\
  & \leq \frac{\pi \rho}{R_k(t - \tau)} \bo 1 \{ \rho \leq R_k(t - \tau) \} + \pi \bo 1 \{ \rho > R_k(t - \tau) \} \notag \\
  & \leq \frac{\pi \rho}{R_k(t - \tau)}, \label{E:cond_theta_bound}
\end{align}
where $\rho = \sup_{u \leq \tau} \max_{1 \leq n \leq N} \|X_n(t - \tau + u) - X_k(t - \tau)\|$, which on the event $\{\tau \leq \Delta\}$ can be dominated by
\[
 \rho \prec \sup_{s \leq \Delta} \sup_n \| \bar Z_n(s) \|,
\]
where $(\bar Z_n(u), 1 \leq n \leq \bar N(u))$ is a $d$-dimensional branching Brownian motion started from one particle at $0$. Writing for all $n$ and $u$,
\[
 \bar Z_n(u) = (\bar Z^{(1)}_n(u), \ldots, \bar Z^{(d)}_n(u)) \in \R^d,
\]
we have that $\bar Z^{(1)}_n(s), \ldots, \bar Z^{(d)}_n(s)$ are one-dimensional branching Brownian motions and so
\begin{align}
 \P( \rho > \sqrt{2} d \Delta + C_2 d \log t, \ \tau \leq \Delta) & \leq 2d \P(\sup_{s \leq \Delta} \sup_n \bar Z^{(1)}_n(s) > \sqrt{2} \Delta + C_2 \log t) \notag \\
  & \leq 4de^{-\sqrt2 C_2 \log t}, \label{E:rho_tail}
\end{align}
where \eqref{E:rho_tail} follows by Lemma \ref{L:maximum}. This is also summable for sufficiently large $C_2$, hence we deduce that almost surely there exists $T_2 > 0$, possibly random, such that if $t > T_2$ then $\tau \leq C_1 \log t$ and $\rho \leq C_3 \log t$ where $C_3 = \sqrt2 d C_1 + d C_2$.

Then for $t > T_2$, applying \eqref{E:cond_theta_bound} and Lemma \ref{L:modulus_R}, there exists some $C_4 > 0$ such that
\[
 D(\Theta_m(t),\Theta_n(t)) \leq \frac{\pi C_3 \log t}{R_k(t - \tau)} \leq \frac{C_4 \log t}{v_N (t - C_1 \log t)}.
\]
The right hand side tends to 0 as $t \to \infty$ uniformly over $m,n$, so almost surely
\[
 \sup_{1 \leq m,n \leq N} D(\Theta_m(t), \Theta_n(t)) \to 0,
\]
as desired.
\end{proof}

Note that by Lemma \ref{L:geo_MRCA_bound}, the system eventually has a unique most recent common ancestor. We also observed that almost surely, the time of the most recent common ancestor $t - \tau(t) \to \infty$ as $t \to \infty$. If we consider the genealogical path of the most recent common ancestor, we see that there is a unique immortal genealogical path in the system, or the `spine', from which all the particles that are eventually ever alive in the system descend from.

Let $X_{*}(t)$ be the particle of the spine at time $t$ and take the usual decomposition: $X_{*}(t) = R_{*}(t)\Theta_{*}(t)$ where $R_{*}(t) > 0$ and $\Theta_{*}(t) \in \bb S^{d-1}$ is continuous. We now complete the proof of Theorem \ref{T:modulus} by showing the angular part of the spine converges.

\begin{prop}\label{P:theta_conv}
For all $N > 1$, $\Theta_{*}(t)$ converges almost surely as $t \to \infty$.
\end{prop}

We offer two proofs of this proposition. One is shorter but relies explicitly on stochastic calculus, and hence works only for the exact situation described in this paper. On the other hand the second proof is a bit longer but more robust; in particular it carries over to slightly more general Brunet--Derrida particle systems than the ones we consider in this paper: see Remark \ref{R:gen}. 

\begin{proof}[First proof of Proposition \ref{P:theta_conv}] We start by recalling the classical \emph{skew-product decomposition} of Brownian motion (see for example Section 7.15 of \cite{IM}). The version we present here is Theorem 1.1(d) of \cite{RVY}.

Let $(X(t), t \geq 0)$ be a $d$-dimension Brownian motion, and write $X(t) = R(t) \Theta(t)$ with $R(t) > 0$ and $\Theta(t) \in \mathbb{S}^{d-1}$ and $R, \Theta$ continuous. Let
\begin{equation}\label{timechange}
H_t = \int_0^t R(s)^{-2} \, ds.
\end{equation}
Then,
\begin{enumerate}[(i)]
 \item $(R(t), t \geq 0)$ is a Bessel process of order $d$.
 \item Under the time change $\Phi(H_t) = \Theta(t)$, $(\Phi(t), t \geq 0)$ is a Brownian motion on $\bb S^{d-1}$.
 \item $(\Phi(t), t \geq 0)$ is independent of $(R(t), t \geq 0)$.
\end{enumerate}
In the special case of $d = 2$, we can write $\Theta(t)$ as $e^{iB(H_t)}$ where $(B(t), t \geq 0)$ is a standard Brownian motion in $\R$ indepedent of $(R(t), t \geq 0)$. We note here that in this case $\Phi(t) = e^{iB(t)}$.

Now consider the system $(\bar X(t), t \geq 0)$ that results from not enforcing selection: the underlying free $d$-dimensional branching Brownian motions started from $N$ particles at $X_1(0), \ldots, X_N(0)$ coupled to the Brunet--Derrida system. It is clear that this can be constructed by considering the skew product decomposition of every Brownian path in the system $ X_i(t) =  R_i(t) \bar \Theta_i(t) = \bar R_i(t) \bar \Phi_i(H_i(t))$. This is a bit cumbersome, but here are the details.

Let $\cT$ be the underlying branching tree (which by assumption is just an ordinary Yule process). We use Neveu's formalism for binary trees, i.e., $\cT$ is a set of vertices given by $\cT = \cup_{n=0}^{\infty} \{0,1\}^n$ and each vertex $v$ has attached to it an independent exponential random variable of mean 1, $X_v$, representing the lifetime of this individual. We call $[s_v, t_v]$ the interval of time over which this particle is alive, thus $t_v - s_v = X_v$ and so $s_v = \sum_{w\preceq v} X_w$ (with $w\preceq v$ means $w$ is ancestor of $v$). We also attach to each $v$ a Bessel process $R_v(t)$ defined over the interval of time $[s_v, t_v]$ in the natural way, by solving the SDE
$$
dR_v(t) = dB_v(t) + \frac{d-1}{2 R_v(t)}dt, t \in [s_v, t_v]
$$
where the Brownian motions $B_v$ are independent for different vertices $v$, and by requiring continuity of the resulting Bessel process when we move up along the branches of the tree. We extend the definition of $R_v(t)$ to the entire interval $[0, t_v]$ simply by defining $R_v(s)= R_w(s)$ where $w$ is the unique ancestor of $v$ alive at time $s$ (i.e., such that $s \in [s_w, t_w]$).

We further enrich this structure by associating to each vertex $v$ an angle process $\Theta_v(t)$, also defined over the interval of time $[s_v, t_v]$, which is defined by applying the construction \eqref{timechange} in between two successive branching events. More precisely, let $$
H_v(t) = \int_{0}^{t} R_v(s)^{-2} ds,$$
let $s'_v = H_v(s_v)$ and $t'_v = H_v(t_v)$. Consider a family of Brownian motions on $(\Phi_v(t), t \in [s'_v, t'_v], v \in \cT)$ on $\mathbb{S}^{d-1}$ such that the evolution of $\Phi_v$ over $[s'_v, t'_v]$ are independent for different vertices $v \in \cT$. As above we extend $\Phi_v(t)$ to the interval $[0, t'_v]$ by defining $\Phi_v(t) = \Phi_w(t)$ where $w$ is the unique ancestor of $v$ such that $t \in [s'_v, t'_v]$, and we have chosen $\Phi_v$ so that $\Phi_v(t)$ is a continuous function of $t$ over $[0, t'_v]$ for all $v \in \cT$. We now define $\Theta_v$ by the formula $$\Theta_v( t) = \Phi_v(H_v(t)),$$ for $s_v \le t \le t_v$. 

Let $S(t)$ be the set of particles alive at time $t$, i.e., the set of vertices $v\in \cT$ such that $t \in [s_v, t_v]$. Let $N(t) = |S(t)|$ and order the vertices in $S(t)$ by $v_1, \ldots , v_{ N(t)}$ in such a way that $R_1(t) \ge R_2(t) \ge \ldots$, where $R_i(t) = R_{v_i}(t)$. Also let $\Theta_i(t) = \Theta_{v_i}(t)$ for $1\le i \le  N(t)$. Then our system of branching Brownian motion then consists of 
$$ X_i(t) =  R_i(t) \Theta_i(t), 1\le i \le N(t), t \ge 0.$$

Having described the skew product decomposition of a free branching Brownian motion $( X_i(t), t \ge 0, 1 \le i \le N(t))$, we proceed with the proof of Proposition \ref{P:theta_conv}. By a \emph{ray} we mean a sequence $V = \{v_1, v_2, \ldots\} $ such that $v_n$ is in generation $n$ of the tree and $v_n \preceq v_{n+1}$ for all $n\ge 0$. For each given ray $V$, we can follow the trajectory $X_V(t)$ of the Brownian motion associated with $V$, that is, $X_V(t) = X_v(t)$ for the a.s. unique $v \in V$ such that $t \in [s_v, t_v]$. We can also consider $R_V(t) = R_v(t)$ its radial part and $\Theta_V(t) = \Theta_v(t)$ its angular part. Observe then that we have, by construction, $\Theta_V(t) = \Phi_V(H_V(t))$ where $\Phi_V$ is a Brownian motion on $\mathbb{S}^{d-1}$ and $H_V(t) = \int_0^t  R_V(s)^{-2} \, ds$. 

Now, consider the set $\mathcal V$ of rays $V$ such that  
$$
\liminf_{t\to \infty} \frac{\| X_V(t) \| }{t} \ge v_N /2.
$$ 
If $V \in \mathcal{V}$, $H_V(t) = \int_0^t  R_V(s)^{-2} \, ds$ converges almost surely as $t \to \infty$ to a limit $H_V(\infty)$. Hence $\Theta_V(t)$ converges as $t \to \infty$ to $ \Phi_V(H_V(\infty))$.
By Lemma \ref{L:modulus_R}, $X_{*}(t)$ almost surely is such a path in $\mathcal V$ and so $\Theta_{*}(t)$ converges as $t \to \infty$ to a limit. 
\end{proof}

\begin{proof}[Second proof of Proposition \ref{P:theta_conv}]
Our second proof relies on a suitable martingale argument rather than stochastic calculus, and hence is more robust. See Remark \ref{R:gen} for a discussion of the setups to which it carries.  Consider a free branching Brownian motion $\bar X = (\bar X_i(t), 1\le i \le \bar N(t), t\ge 0)$, and write $X_i(t) = R_i(t) \Theta_i(t)$ for $t\ge 0$ and $1\le i \le \bar N(t)$. Let $\mathcal F^R_t = \sigma(R_i(u), 1 \le i \le \bar N(u), u \leq t)$ and let $\mathcal F^\Theta_t = \sigma(\Theta_i(u), 1 \le i \le \bar N(u), u \le t)$. Let $\mathcal G_t = \sigma(\mathcal F^R_\infty \cup \mathcal F^\Theta_t)$, and note that $(\theta_*(t), t \ge 0)$ is adapted to the filtration $(\cG_s,s \ge 0)$. 

 We start by explaining the argument in the case $d=2$, which is a bit simpler to describe. Recall in the case $d = 2$, we can write $X_{*}(t) = R_{*}(t) e^{i \theta_{*}(t)}$, where $R_{*}(t) >0$ and $\theta_{*}(t)$ is a continuous function. This way of writing $X_{*}(t)$ is unique modulo a global constant multiple of $2\pi$ in $\theta_{*}(t)$, which we fix once and for all at time 0. 
 
\begin{lem}
$(\theta_{*}(t), t \ge 0)$ is a martingale with respect to $(\cG_t, t \ge 0)$. 
\end{lem}

\begin{proof}
It is a simple exercise left to the reader to check that $\theta_*(s)$ is integrable. Let $s>0$ and suppose at time $s$ there is a particle at position $z$ in this process, say $X_i(s) =z$. Consider the transformation $T= T_z$ which is a reflection in the line $\R z$: 
\[
 T_z(x) =2\< x, z'\> z' - x, x \in \R^d
\]
where $z' = z / \| z\|$. Note that $T_z$ is an orthogonal transformation and hence Wiener is invariant under $T_z$. We apply $T_z$ to every descendant of the particle $X_i(s)$, and call $T(\bar X) = (T(\bar X_i(t) ), 1\le i \le \bar N(t), t\ge s)$ the resulting  transformation of all the particles in the branching Brownian motion. We note that since each $T_z$ leaves Brownian motion invariant, $T(\bar X)$ has also the law of a free branching Brownian motion. Moreover, $T_z$ is an isometry so we have $\|T(\bar X_i(t))\| = R_i(t)$ for all $t \ge 0$ and all $1\le i \le \bar N(t)$. In particular, a particle $T(\bar X_i(t))$ survives the selection procedure if and only if its mirror image $\bar X_i(t)$ does. In particular, the branching times and tree structure of the system are invariant under $T$.

These two properties imply that, conditional on $\mathcal G_s$, $T(\bar X)$ has the same distribution as $\bar X$. On the other hand, observe that if $\bar X(t)$ has a particle at $x$ descending from a particle at $z$ at time $s$, then
$$
\arg T(x) = \arg T_z(x) = 2\arg z - \arg x.
$$
Applying this to $z = X_{*}(s)$ and $x = X_{*}(t)$ shows that
\[
 \Ex[\theta_*(t) - \theta_*(s) | \cG_s] = \Ex[\theta_*(s) - \theta_*(t)|  \cG_s] = 0,
\]
as desired. \end{proof}

When $d \geq 3$, it is necessary to first project onto a two-dimensional subspace $\Pi$ before applying a similar reasoning. 
Let $\Pi$ be a given such plane and let $p_\Pi$ be the orthogonal projection onto $\Pi$. For some fixed $e \in \Pi$ and $x \in \R^d$, define $\arg_\Pi(x)$ to be the continuous directed angle between $p_\Pi(x)$ and $e$.

\begin{lem}
$(\arg_{\Pi}(X_{*}(t)), t \ge 0)$ is a martingale with respect to $(\cG_t, t \ge 0)$. 
\end{lem}

For $z \in \R^d$, let $T_z(x)$ be defined by
\[
 T_z(x) = 
  x - 2\big(p_\Pi(x) - \< p_\Pi(x), z'\>z'\big), x \in \R^d,
\]
where $z' = p_\Pi(z) / \| p_\Pi(z) \|$. More descriptively, if $x = u + v$ where $u \in \Pi$ and $v $ is orthogonal to $\Pi$, then $T_z(x) = T_z(u) + T_z(v)$, where $T_z(v) = v$ and $T_z(u)$ is the reflection of $u$ in the line $\R P_{\Pi}(z)$ within the plane $\Pi$. As before, applying this transformation to each descendant of a particle located at $z$ at time $s$ yields a transformation $T$ of the branching Brownian motion, which leaves the modulus of particles $\|T(\bar X_i(t)) \| = \|\bar X_i(t)\|$ unchanged, and leaves the law of branching Brownian motion also unchanged. But the choice of $T$ gives $\arg_\Pi(T(x)) = 2\theta_\Pi(z) - \arg_\Pi(x)$ if $x$ descends from $z$. Thus
\[
 \Ex[\arg_{\Pi}(X_{*}(t)) - \arg_{\Pi}(X_{*}(s))| \mathcal G_s] = \Ex[\arg_{\Pi}(X_{*}(s)) - \arg_{\Pi}(X_{*}(t))| \mathcal G_s] = 0,
\]
as above. This concludes the proof of the lemma.
\qed

\medskip We are now ready to conclude the second proof of Proposition \ref{P:theta_conv}. It suffices to prove that $\theta_*(t)$ converges as $t \to \infty$. We can assume without loss of generality that $d=2$, as it suffices to show that $\arg_\Pi (X_*(t))$ converges as $t \to \infty$ for any fixed arbitrary two-dimensional subspace $\Pi$. Thus we will assume $d=2$.

Let $t > 0$ and set $s = \lceil t \rceil - 1$. Define $\rho_*(t) = \sup_{u \in [s,t]} \|X_{*}(u) - X_{*}(s)\|$ and define a stopping time $T$ to be the first time $t$ such that $\rho_*(t) \geq R_{*}(s)$. Let $\theta^T_*(t) = \theta_*(t\wedge T)$ be the martingale $\theta_*(t)$ stopped at $T$. The reason for stopping at $T$ is to ensure a bound similar to \eqref{E:trig} holds for $\theta_*^T(t)$. The precise bound is
\begin{equation}\label{E:step_bound}
 |\theta^T_*(s + 1) - \theta^T_*(s)| \leq \frac{\pi \rho_*(s+1)}{2 R_{*}(s)} \leq \frac\pi2.
\end{equation}
Since $\theta^T$ is a martingale, 
\begin{align}
 \Ex[\theta^T_*(t)^2] & = \sum_{s = 0}^{t-1} \Ex[(\theta^T_*(s + 1) - \theta^T_*(s))^2]   \le \sum_{s = 0}^{t-1}     \frac{\pi^2}4 \E\left[1\wedge \frac{ \rho_*(s+1)^2}{R_{*}(s)^2} \right]  
 \label{ubTheta}
\end{align}
We observe that $R_*(s) \ge  \| X_N(s) \|$ and proceed to bound from above $\|X_N(s)\|$ stochastically. 
Using the coupling used in the proof of Lemma \ref{L:modulus_R}, we have that $\|X_N(t)\|$ dominates the minimum at time $t$ of a standard one-dimensional Brunet--Derrida system started from $N$ particles all at the origin. In turn, this dominates $S(t) = Z(1) + \ldots + Z(t)$ where $Z(i)$ are independent and identically distributed as the position of the minimum at time $1$ of a one-dimensional Brunet--Derrida system started from $N$ particles all at the origin by the monotone coupling of Lemma \ref{L:monotone}. 

Let $m =\Ex[Z(1)] $, and note that for $N>1$, $m>0$ for the same reason $v_N > 0$ in Lemma \ref{L:1D_speed}. Furthermore, $Z(1)$ is the minimum of a finite number of Brownian motions at time 1 and hence $\Ex[e^{-\lam Z(1)}] < \infty$ for all $\lam \geq 0$ and so $\psi(\lam) = \log \Ex[e^{-\lam Z(1)}]$ is well-defined. Then for any $\lam \geq 0$,
\begin{align*}
 \textstyle \P\left(S(t) \leq \frac12mt\right) & \leq \P\left(e^{-\lam S(t)} \geq e^{-\frac12 \lam mt}\right) \\
 & \leq e^{\frac12 \lam mt}\Ex[e^{-\lam S(t)}]  \leq \exp(tf(\lam)),
\end{align*}
where $f(\lam) = \frac12 \lam m + \psi(\lam)$. We note that $f(0) = 0$ and $f'(0) = \psi'(0) + \frac12 m = -\zeta m$ and that for $\lam$ sufficiently small, $f(\lam) \leq \frac12 \lam f'(0)$. Therefore, 
\[
 \textstyle \P(S(t) \leq \frac12 mt) \leq \exp(-\frac14 \lam mt),
\]
Therefore, by Jensen's inequality and since $x \mapsto x \wedge 1$ is concave,
\begin{align*}
\E\left[ 1 \wedge \frac{\rho_*(s+1)^2}{\|X_N(s)\|^2} \right] & \le \E\left[\left( 1\wedge  \frac{4\rho_*(s+1)^2}{m^2 s^2} \right) \bo 1 \{ \|X_N(s)\| \ge ms /2\} \right] + \P( \|X_N(s)\| \le ms /2 )\\
& \le 1\wedge \frac{4\E[\rho_*(s+1)^2]}{m^2 s^2} + e^{-cs} 
\end{align*}
It is not hard to see that there exists $C_1 > 0$ depending on $N$ but not $s$ such that $\E[\rho_*(s+1)^2] \leq C_1$. Therefore, plugging into \eqref{ubTheta} we see that 
\[
 \E[\theta^T_*(t)^2] \le C_2
\]
for some $C_2 > 0$ and so $(\theta_*^T(t), t\ge 0)$ is a martingale bounded in $L^2$, and so converges almost surely. 

Obviously, this implies convergence of $\theta_*$ almost surely on the event $\{T=\infty\}$. Therefore it suffices to check that that almost surely $\rho_*(t) \geq R_{*}(s)$ eventually never happens. But note that since $\E[\rho_*(t)^2] \le C_1 < \infty$ it follows from Markov's inequality and the Borel--Cantelli lemma that $\rho_*(t) < v_N (\lceil t \rceil - 1) /2$ for all $t$ sufficiently large, and hence $\rho_*(t) \le R_*(s)$ for all $t$ sufficiently large by Lemma \ref{L:modulus_R}. Thus $\theta_*(t)$ converges almost surely as $t \to \infty$.
\end{proof}

\begin{rmk}\label{R:gen}
In this paper we have concerned ourselves for simplicity with branching Brownian motion with selection. However, there are a variety of possible alternatives: for instance, initially, Brunet and Derrida considered a system where branching occurs at discrete time steps $t = 0,1, \ldots$, and at each $t$, each particle branches into two (or possibly even more) individuals,  and the displacement follows a random walk with a given distribution. Yet another alternative, taken up by Durrett and Remenik, is to have particles branch at rate 1 in continuous time. 

As is plain from the above proof, Theorem \ref{T:modulus} remains true in each of these cases, under the assumption that the displacement of particles is rotationally symmetric and second moment on the random walk jumps.
\end{rmk}

\section{Proof of Proposition \ref{P:min_upper} }

Consider a standard one-dimensional Brunet--Derrida particle system with $N$ particles, started from an initial configuration satisfying \eqref{E:init_cond}. For ease of notation, we will assume without loss of generality (since the system is translation invariant) that $x = 0$. 

The key idea of the proof is to compare the Brunet--Derrida system to the free branching Brownian motion where killing occurs at a linear boundary. This is the idea which lies behind papers such as \cite{BBS}, which used a wall of velocity $\sqrt{2 - 2 \pi^2 / (\log N)^2}$ in first approximation. We shall use the same speed for our wall, which we will call the \emph{right wall}. More precisely, we define $L = (\log N)/\sqrt2$, for some $\nu > 1$, $\sqrt \e = \pi / L$ and $\mu^2 = 2 - \e$ and consider a moving linear boundary $(L + \mu t, t \geq 0)$.

There is a natural coupling to a free branching Brownian motion in $\R$ with particles $\bar X_i(t), 1\le i \le \bar N(t)$, ordered in the usual way right to left, obtained by ignoring any particle with index greater than $N$ and their descendants. Note the key property of this coupling that for each $1 \le i \le N$, $X_i(t) \le \bar X_i(t)$, with probability one. Therefore, under this coupling,
\begin{equation}\label{E:minfree}
 \P(X_N(t) \geq \mu t) \leq \P(\bar X_N(t) \geq \mu t). 
\end{equation}

We note here that our initial condition \eqref{E:init_cond} allows us to prove the Proposition up to a constant (that does not depend on $N$) shift of the system. In other words, for any fixed $\zeta > 0$, it suffices to show
\[
 \P \left( \sup_t \left\{ X_N(t) - \mu t \right\} \geq \zeta \right) \to 0,
\]
since
\[
 \sum_{i = 1}^N e^{\sqrt2 (X_i(0) + \zeta)} = e^\zeta \sum_{i = 1}^N e^{\sqrt2 X_i(0)} \leq e^\zeta N^\delta,
\]
which also satisfies \eqref{E:init_cond} for $N$ sufficiently large.

Let $T = c_\delta(\log N)^3$ where $c_\delta > 0$ is a small constant depending only on $\delta$ which we will fix later on. The idea is that $c_\delta$ will be small enough so that the event $V_t$ that no particle ever hits $L + \mu t$ up to time $t$, has high probability for any $t \le T$. (See Lemma \ref{L:kill_right}).

In view of this, let $I(t) \subset \{1, \ldots, \bar N(t)\}$ denote the index set of particles that never touch position $L + \mu s$ for any $s \leq t$. This corresponds to killing particles once they hit this position. Note that $V_t = \{\# I(t) = \bar N(t)\}$. Let
\[
 W_t \defeq \sum_{i \in I(t)} \bo 1 \{\bar X_i(t) \ge \mu t\}
\]
be the number of particles of the free branching Brownian motion with killing at $L + \mu t$, $t \le T$, which are greater or equal to $\mu t$.
Then by \eqref{E:minfree}, we get
\begin{equation}\label{E:minfreekill}
 \P(X_N(t) \geq \mu t) \le \P(W_t \ge N) + \P(V_t^\complement) \le \frac{\Ex[W_t]}{N} + \P(V_t^\complement)
\end{equation}
by Markov's inequality.

In order to estimate $W_t$ we consider an additional wall (which we call the \emph{left wall}) which also moves at velocity $\mu$, and starts at position $0$. We will treat separately the particles in $I(t)$ that hit the left wall and those that do not. More precisely, let $J(t) \subset I(t)$ denote the index of particles that never touch position $L + \mu s$ or $\mu s$ for any $s \le t$. Thus the particles in $J(t)$ are killed when they hit either of two walls (the left and the right one) which both move at velocity $\mu$, and start at position $0$ and $L$. 

Let
\begin{align*}
 W_1(t) & = \sum_{i \in J(t)} \bo 1 \{\bar X_i(t) \ge \mu t \}, \\
 W_2(t) & = \sum_{i \in K(t)} \bo 1 \{\bar X_i(t) \ge \mu t\},
\end{align*}
where $K(t) \defeq I(t) \setminus J(t)$, so that
\[
 W_t = W_1(t) + W_2(t).
\]

\begin{lem} \label{L:W1}
There exists a universal constant $C > 0$ such that for any $t \le T$,
\[
  \Ex[W_1(t)] \le \sum_{i : \bar X_i(0) > 0} C e^{\mu \bar X_i(0)}.
\]
\end{lem}

\begin{proof}
We apply a result due to Maillard (second part of Lemma 5.4 in \cite{Maillard}). 
Let $r \in (0,L)$ and assume that initially there is one particle at $x > 0$. Then the number of descendants $N_x(t)$ of that particle that do not hit either  the left or right wall, and which lie in $[r+\mu t , L + \mu t]$, satisfies 
\[
 \Ex[N_x(t)] \leq  C e^{\mu (x - r)},
\]
for some universal constant $C>0$. Letting $r\to 0$ by monotone convergence theorem and summing over all initial positions $x>0$ of particles gives the result.
\end{proof}

It remains to treat particles that do hit the left wall. Let $\Delta_t$ be the number of particles that are killed on the left wall if we kill all particles that hit this wall.

\begin{lem}\label{L:W2Delta}
Given $\Delta_t$, 
\[
 \Ex[W_2(t) | \Delta_t] \le e^{\frac12 \e t} \Delta_t.
\]
\end{lem}

\begin{proof}
Let $\mathcal F_t = \sigma(\Delta_s, s \le t)$. For each particle killed on the left wall at some time $s\le t$, the conditional expectation, given $\mathcal F_t$, of the number of descendants at time $t$ that are greater or equal to $\mu t$ is simply, by translation invariance, and the many-to-one lemma,
\begin{equation}\label{E:cond_exp}
e^{t-s} \P(B_{t-s} \ge \mu (t-s)).
\end{equation}
 Now, let 
\[
 f(t) = e^t \int_{\mu t}^{\infty} e^{-\frac{x^2}{2t}} \frac{dx}{\sqrt{2\pi t}}.
\] 
A quick calculation shows that since $\mu < \sqrt{2}$, $f'(t) \ge 0$ so \eqref{E:cond_exp} is maximised at $s=0$. Thus, summing over all the times $s$ at which some particle dies touching the left wall (thereby increasing $\Delta_s$ by one), yields
\[
\Ex \left[ \left. \sum_{i \in K(t)} \bo 1 \{\bar X_i(t) \ge \mu t\} \right| \mathcal F_t \right] \le e^t\P(B_t \ge \mu t) \Delta_t.
\] 
Since $\Delta_t$ is $\mathcal F_t$-measurable, the result follows, after noticing that $e^t \P(B_t \ge \mu t) \le e^t e^{ - \frac12 \mu^2t } = e^{\frac12 \e t} $ by well known bounds on the normal distribution tail.
\end{proof}

Hence we have reduced the problem to estimating from above $\Ex[\Delta_t]$. To this end, we will distinguish between those that started at positive positions and those at negative positions, respectively $K_+(t)$ and $K_-(t)$. Call $\Delta_+(t)$ and $\Delta_-(t)$ the corresponding number of particles killed at the left wall.

\begin{lem}\label{L:Delta+}
\[
 \Ex[\Delta_+(t)] \le C e t \sum_{i : \bar X_i(0) > 0} e^{\mu \bar X_i(0)}
\]
\end{lem}

\begin{proof}
Any particle that first hits the left wall from the right has to descend from an ancestor $\bar X_i(0)$ at time $0$ with $\bar X_i(0) > 0$. We use the following very crude bound: during times $t$ and $t+1$, given $W_1(t)$,
\begin{equation}\label{increaseDelta}
 \Ex[\Delta_+(t+1) - \Delta_+(t) | W_1(t)] \le e W_1(t).
\end{equation}
This is because the number of particles killed on the left wall during $[t, t+1]$ cannot exceed the total number of descendants of particles $i \in J(t)$ such that $X_i(t) >0$. Since the number of such particles is precisely $W_1(t)$, \eqref{increaseDelta} follows.

Taking expectations in \eqref{increaseDelta} we have by Lemma \ref{L:W1},
\[
 \Ex[\Delta_+(t+1) - \Delta_+(t)] \le e \Ex[W_1(t)] \leq C e \sum_{i : \bar X_i(0) > 0} e^{\mu \bar X_i(0)},
\]
Since $\Delta_s$ is non-decreasing, the lemma follows by summing over the intervals $[0,1], \ldots, [\lfloor t \rfloor, \lceil t \rceil]$.
\end{proof}

We now address $\Delta_-(t)$.
\begin{lem}\label{L:Delta-}
\[
 \Ex[\Delta_-(t)] \leq e^{\frac12 \e t} \sum_{i : \bar X_i(0) < 0} e^{\mu \bar X_i(0)}.
\]
\end{lem}

\begin{proof}
Any particle that first hits the left wall from the left has to descend from an ancestor $\bar X_i(0)$ at time $0$ with $\bar X_i(0) < 0$. The total number of such particles up to time $t$, $\Delta_-(t)$, is exactly the number of particles of a branching Brownian motion with drift $-\mu$ that hit level $0$ by time $t$, started from the negative positions in the initial condition. 

Fix some constant $A > 0$, and consider a branching Brownian motion with drift $-\mu$ where every particle is \emph{stopped} upon reaching 0 and \emph{killed} upon reaching $-A$. Initially the starting positions consists precisely of $(\bar X_i(0), 1\le i \le N)$ whenever $\bar X_i(0) < 0$. We call $X^*_i(t)$, $i \in N^*(t)$, the corresponding particle locations. Let $\Delta_-^A(t)$ be the number of particles stopped upon reaching $0$ by time $t$. Now consider the process
\begin{equation}\label{E:supermartingale}
 M^A_s = \sum_{i \in N^*(t)} (X^*_i(s) + A) e^{\mu (X^*_i(t) + A) - \frac12(2 - \mu^2)s}.
\end{equation}
Without stopping particles upon reaching $0$, it is easy to check that $(M^A_s, s \geq 0)$ defines a nonnegative martingale (see e.g. Lemma 2 of \cite{HH}, or Lemma 6 of \cite{BBS})). However, if we stop particles upon reaching $0$, since $2 - \mu^2 = \e > 0$, $M^A_s$ becomes a supermartingale. Therefore
\[
 \sum_{i : \bar X_i(0) < 0} A e^{\mu (\bar X_i(0) + A)} \geq \Ex[M^A_0] \geq \Ex[M^A_t] \geq \Delta_-^A(t) Ae^{\mu A - \frac12\e t}.
\]
So, making the cancellations,
\[
 \Ex \left[ \Delta_-^A(t) \right] \leq e^{\frac12 \e t} \sum_{i : \bar X_i(0) < 0} e^{\mu \bar X_i(0)}.
\]
Letting $A \to \infty$ and using the monotone convergence theorem concludes the proof of Lemma \ref{L:Delta-}.
\end{proof}

With this supermartingale argument, we are also in a position to address $V_t$, the event no particle ever hits $L + \mu t$ up to time $t$.

\begin{lem}\label{L:kill_right}
\[
 \P(V_t^\complement) \leq e^{-\mu L} e^{\frac12 \e t} \sum_{i = 1}^N e^{\mu \bar X_i(0)}.
\]
\end{lem}

\begin{proof}
We use the same supermartingale \eqref{E:supermartingale} as in the proof of Lemma \ref{L:Delta-} except we now stop particles upon reaching $L$ in the branching Brownian motion with drift $-\mu$. Particles are still killed at $-A$. Let $\Delta_L^A(t)$ be the number of particles stopped upon reaching $L$ by time $t$. Then arguing as before
\[
 \Ex[\Delta_L^A(t)] (L+A) e^{\mu(L+A) - \frac12 \e t} \leq \sum_{i = 1}^N (\bar X_i(0) + A)e^{\mu(\bar X_i(0) + A)}.
\]
Since $L > X_1(0)$,
\[
 \Ex[\Delta_L^A(t)] \leq e^{\frac12 \e t - \mu L} \sum_{i = 1}^N e^{\mu \bar X_i(0)}.
\]
We note that $\P(V_t^\complement) = \P(\lim_{A \to \infty} \Delta_L^A(t) \geq 1)$ and conclude by Markov's inequality and monotone convergence.
\end{proof}

Putting together Lemmas \ref{L:Delta+} and \ref{L:Delta-}, we get
\begin{align*}
 \Ex[\Delta_t] & \leq  \left( e^{\frac12 \e t} \sum_{i : \bar X_i(0) < 0} e^{\mu \bar X_i(0)} + C e t \sum_{i : \bar X_i(0) > 0} e^{\mu \bar X_i(0)} \right).
\end{align*}
Combining with Lemmas \ref{L:W1} and \ref{L:W2Delta}, this yields
\[
 \Ex[W_t] \leq e^{\e t} \sum_{i : \bar X_i(0) < 0} e^{\mu \bar X_i(0)} + (1 + e t)C e^{\frac12 \e t} \sum_{i : \bar X_i(0) > 0} e^{\mu \bar X_i(0)} 
\]
Note that for $t \le T$,
\[
 e^{\e t} \le \exp(\e c_\delta (\log N)^3) = \exp(2\pi^2 c_\delta \log N) = N^{2\pi^2 c_\delta}.
\]
(Recall the definition of $\eps$ at the beginning of the section.) Therefore,
\[
 \Ex[W_t] \leq N^{2\pi^2 c_\delta} \sum_{i = 1}^N e^{\mu \bar X_i(0)},
\]
and because \eqref{E:init_cond} holds,
\[
 \P(W_t \geq N) \leq \frac{\Ex[W_t]}N \leq N^{-\kappa},
\]
where $\kappa = 1 - 2\pi^2 c_\delta - \delta > 0$ for a small enough choice of $c_\delta$. As for $\P(V_t^\complement)$, since
\[
 \mu L = \log N \sqrt{1 - \frac{\pi^2}{(\log N)^2}} > (1 - \pi^2 c_\delta)\log N,
\]
for all $N$ sufficiently large, by Lemma \ref{L:kill_right},
\[
 \P(V_t^\complement) \leq N^{-\kappa}.
\]
By \eqref{E:minfreekill}, we now have a bound for any fixed time $t \leq T$,
\begin{equation}\label{E:max_free_conc}
 \P(X_N(t) \geq \mu t) \leq \P(\bar X_N(t) \geq \mu t) \leq 2N^{-\kappa}.
\end{equation}

We now extend this bound to hold for all $t \leq T$. Let $t_k = i ( \log N)^{-1}, k = 1, \ldots c_\delta(\log N)^2$, so that $t_k$ forms a regular partition of $[0,T]$ with spacings of size $1 / (\log N)$. During each $[t_k, t_{k+1}]$, it is possible to check that $X_N(t)$ has small fluctuations. The key observation here is that $X_N(t)$ is piecewise Brownian and only jumps to the right -- never to the left. Therefore, during the interval, the minimum cannot travel too far right of $\mu t_{k+1}$ due to the cost of motion to the left.

We now fix some large constant $K > 0$ and define the bad events
\begin{equation}
 B_k = \left\{ \sup_{s \in [t_{k-1}, t_k]} X_N(s) \geq \mu t_k + K \right\},
\end{equation}
and the good events $G_k = \{ X_N(t_k) \leq \mu t_k \}$.

Given the Brunet--Derrida system at time $t_{k-1}$, consider the coupled free branching Brownian motion started from these $N$ particles and for a particle $\bar X_i(t_k)$ at time $t_k$, let $\bar Y_i(s)$ be its ancestor at time $s \leq t$. We see by the observation above, the probability of the event $B_k \cap G_k $ is bounded by the probability of the analogous event for the branching Brownian motion, namely  
\[
\left\{\exists 1\le i \le \bar N(t_k), \sup_{s \in [t_{k-1}, t_k]} \bar Y_i(s) \geq \mu t_k + K , \quad \bar X_i(t_k) \leq \mu t_k \right\}
\]
By a union bound and the many-to-one lemma (Lemma \ref{L:many_to_one}),
\begin{align*}
 \P(B_k \cap G_k)  
 & \leq N e^{t_k - t_{k-1}} \P \left( \sup_{s \in [t_{k-1}, t_k]}  Z(s) - Z(t_k) \geq K \right) \\
 & \leq 2 N e^{(\log N)^{-1}} \P(Z((\log N)^{-1}) \geq K) \\
 & \leq 4 N e^{-\frac12 K^2 \log N}
\end{align*}
where $(Z(u), u \geq 0)$ is a Brownian motion. So for $K > \sqrt {2(1 + \kappa)}$,
\begin{equation}\label{E:fluc_min}
 \P(B_k \cap G_k) \leq 4 N^{-\kappa}.
\end{equation}
We can sum the conclusion of \eqref{E:max_free_conc} and \eqref{E:fluc_min} over all $k$ to show that
\begin{align*}
 \P \left( \sup_{t \leq T} \left\{ X_N(t) - \mu t \right\} \geq K + \mu(\log N)^{-1} \right) 
 & \leq \sum_{k=1}^{c_\delta(\log N)^2} ( \P(B_k \cap G_k) + \P(G_k^\complement)) \\
 & \leq 6 c_\delta (\log N)^2 N^{-\kappa} \to 0.
\end{align*}
\qed

\section{Proof of Theorem \ref{T:MRCA}}

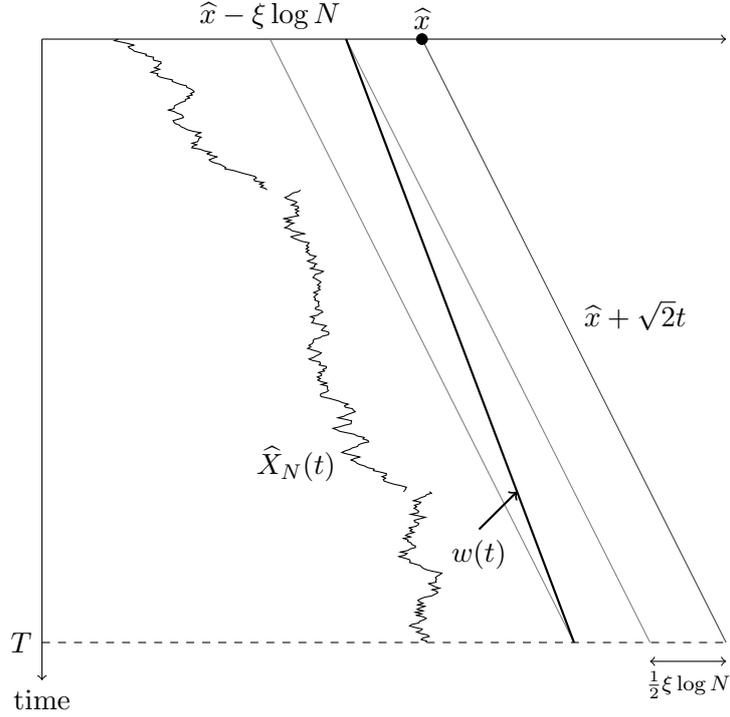
\begin{figure}
\centering 
\begin{tikzpicture}
 \draw [<-] (-2,-0.5) node [below] {time} -- (-2,8);
 \draw [->] (-2,8) -- (7,8);
 \draw [dashed] (-2,0) node [left] {$T$} -- (7,0);
 \draw [darkgray] (3,8) -- (7,0);
 \node [above right] at (5,4) {$\hat x + \sqrt 2 t$};
 \draw [gray] (2,8) -- (6,0);
 \draw [gray] (1,8) -- (5,0);
 \node [above] at (1,8) {$\hat x - \xi \log N$};
 \draw [fill] (3,8) circle [radius=2pt] node [above] {$\hat x$};
 \draw [<->] (6,-0.25) -- (7,-0.25);
 \node [below] at (6.5,-0.25) {$\scriptstyle \frac12 \xi \log N$};
 \draw [thick] (2,8) -- (5,0);
 \draw [thick, ->] (3.75,1.5) -- (4.25,2);
 \node [below] at (3.75,1.5) {$w(t)$};
 \pgfmathsetseed{36};
 \draw [thin] (-1,8)
  \foreach \x in {1,2,...,100} { -- ++(rand*0.1+0.02,-0.02) };
 \draw [thin] (1.4,6)
  \foreach \x in {1,2,...,200} { -- ++(rand*0.1,-0.02) };
 \draw [thin] (3.1,2)
  \foreach \x in {1,2,...,100} { -- ++(rand*0.1,-0.02) };
 \node [above left] at (2,2) {$\hat X_N(t)$};
\end{tikzpicture}
\caption{Diagram reference for proof of Theorem \ref{T:MRCA}.}
\label{F:MRCA}
\end{figure}

Consider a Brunet--Derrida particle system with $N$ particles, started from an initial configuration satisfying \eqref{E:init_cond}, driven by the linear score function $s(x) = \< x, \lam \> = \hat x$. As before, let $T = c_\delta(\log N)^3$ where $c_\delta > 0$ is a small constant depending only on $\delta$ which we will fix later on. Let $\xi > 0$ be small enough that $\delta' \defeq \delta + \sqrt{2} \xi <1$. Note that \eqref{E:init_cond} implies that if $\hat Y_n(t) = \hat X_n(t) -\hat x + \xi \log N$, then
\[
 \sum_{n=1}^N e^{\sqrt{2} \hat Y_n(0)} \le N^{\delta'}
\]
Thus by Proposition \ref{P:min_upper}, with probability tending to $1$, for all $t \leq T$,
\begin{equation}\label{E:MRCA_left_wall}
 \hat X_N(t) \leq \hat x - \xi \log N + \sqrt 2 t.
\end{equation}
On this event,
\begin{equation}\label{w}
 \hat X_N(t) \leq w(t) \defeq \hat x  - \frac12 \xi \log N + \mu' t .
\end{equation}
where
\begin{equation}
\mu' =  \sqrt 2 - \frac{\xi}{2c_\delta(\log N)^2} .
\end{equation}
The function $w(t)$ is a linear boundary which will act as a \emph{killing wall}. Note that
\[
 w(0) = \hat x - \frac12 \xi \log N, \quad w(T) = \hat x - \xi \log N + \sqrt 2T,
\]
and thus by \eqref{w}, if a particle of never hits $w(t)$ and starts to its right (i.e., $\hat X_i(0) \ge w(0)$), then it will survive selection in the Brunet--Derrida system. 

Now, let $Q_{\mu'}(y)$ be the probability that a branching Brownian motion starting from one particle at $y>0$ survives killing at a wall $\mu't$ for all time. Then the probability of a particle at position $y$ or greater to survive in the Brunet--Derrida system until time $T$ is greater than the probability it survives killing at the wall $w(t)$ until time $T$, which is in turn bounded below by $Q_{\mu' }( (\xi/2) \log N)$. 

By Theorem 1 in \cite{BBS2}, we deduce that for $c_\delta < \xi^3/(2 \sqrt 2 \pi^2)$, $Q_{\mu' }( (\xi/2) \log N) \to 1$ as $N \to \infty$, and hence any particle at $x$ has descendants alive at time $T$, as desired. 
\qed

\section{Proof of Theorem \ref{T:sausage}}

Let $H = \{ x \in \R^d: \< x, \lam \> = 0 \}$ be the orthogonal hyperplane to $\lambda$. Recall that when a Brunet--Derrida particle system is driven by a linear $s(x) = \< x, \lam \>$. We have already shown in Lemma \ref{L:diameter} that for any initial condition, if $t > (1 + \kappa) \log N$ for some $\kappa > 0$ and $a > 3\sqrt 2$,
\[
 \lim_{N \to \infty} \P \left( \diam_t \leq a\log N \right) = 1.
\]
On the other hand, by Theorem \ref{T:MRCA}, if \eqref{E:init_cond} initially holds with $x = X_i(0)$ for some $i \geq 2$, both the rightmost and second rightmost particles have descendants alive at time $T$ with high probability, where $T = c_\delta(\log N)^3$, for some $c_\delta > 0$ possibly depending on $\delta$ in \eqref{E:init_cond}.

If \eqref{E:init_cond} initially holds only for $x = X_1(0)$, we observe that the maximum particle at time $u = \log \log N$ fails to branch by time $u$ with probability $\leq 1/(\log N)$. Moreover, by Lemma \ref{L:maximum}, on the event the maximum branches before time $u$, there will be with high probability at least two particles at time $u$ with position $\geq x - 2 \log \log N$. But by introducing an extra $2 \log \log N$ term in \eqref{E:MRCA_left_wall} in the proof of Theorem \ref{T:MRCA}, we see we can apply the conclusion of the theorem to both these particles.

Therefore, at time $u$, two separate particles $X_i(u), X_j(u)$ both have descendants alive at time $T$ with high probability. Let $E$ be this event and call $X_i(T), X_j(T)$ the positions of two arbitrarily chosen descendants of both particles. Let $Y_i(t), Y_j(t)$ denote the positions at time $u \le t \le T$ of the ancestors of $X_i(T)$ and $X_j(T)$. Hence $Y_i(u) = X_i(u)$ and $Y_j(u) = X_j(u)$.

Then note that if $p_H$ is the orthogonal projection onto $H$, $p_H(Y_i)$ and $p_H(Y_j)$ are independent $(d-1)$-dimensional Brownian motions on $H$ on the time interval $[u, T]$ (see the end of the proof of Theorem \ref{T:linear}). Thus
\[
 \diam_t^\perp \succ \|p_H(Y_1(t)) - p_H(Y_2(t)) \|,
\]
on the event $E$, and in particular,
\[
\liminf_{\eta \to 0} \liminf_{N \to \infty} \P \left( \diam_T^\perp \geq \eta (\log N)^{3/2} \right) = 1,
\]
as desired. \qed



\begin{thebibliography}{99}

\bibitem{Bell} G. Bell (1982). \emph{The masterpiece of nature}. Univ. of California Press, Berkeley, CA.

\bibitem{BG} J. B\'erard and J.-B. Gou\'er\'e (2010). Brunet-Derrida behavior of branching-selection partical systems on the line. \emph{Comm. Math. Phys.}, {\bf 298} (2), 323--342.

\bibitem{Ensaios} N. Berestycki (2009). \emph{Recent progress in coalescent theory}. Ensaios Matematicos, vol. 16. Also available at arXiv:0909.3985.

\bibitem{BBS} J. Berestycki, N. Berestycki, and J. Schweinsberg. The genealogy of branching Brownian motion with absorption. \emph{Ann. Probab.}, to appear.

\bibitem{BBS2}  J. Berestycki, N. Berestycki, and J. Schweinsberg. Survival of near-critical branching Brownian motion. \emph{J. Stat. Phys.}, {\bf 143} (5), 833--854, 2011.

\bibitem{BerestyckiYu} J. Berestycki, F. Yu. Unpublished work.

\bibitem{BD1} E. Brunet and B. Derrida (1997). Shift in the velocity of a front due to a cutoff. {\it Phys. Rev. E} {\bf 56}, 2597--2604.

\bibitem{BD2} E. Brunet and B. Derrida (1999). Microscopic models of traveling wave equations. {\it Computer Physics Communications}, {\bf 121--122}, 376--381.

\bibitem{BD3} E. Brunet and B. Derrida (2001). Effect of microscopic noise on front propagation. {\it J. Statist. Phys.} {\bf 103}, 269--282.

\bibitem{BDMM1} E. Brunet, B. Derrida, A. H. Mueller, and S. Munier (2006). Noisy traveling waves: effect of selection on genealogies. {\it Europhys. Lett.}, {\bf 76}, 1--7.

\bibitem{BDMM2} E. Brunet, B. Derrida, A. H. Mueller, and S. Munier (2007). Effect of selection on ancestry: an exactly soluble case and its phenomenological generalization. {\it Phys. Rev. E}, {\bf 76}, 041104.

\bibitem{Burt} A. Burt (2000). Perspective: sex, recombination and the efficacy of selection -- was Weismann right? \emph{Evolution}, {\bf 54} (2), 337--351.


\bibitem{DR} R. Durrett and D. Remenik (2011). Brunet--Derrida particle systems, free boundary problems and Wiener--Hopf equations. \emph{Ann. Probab.}, {\bf 39} (6), 2043--2078.

\bibitem{Etheridge} A. Etheridge (2000). \emph{An Introduction to Superprocesses}, University Lecture Series, vol. 20. American Mathematical Society.

\bibitem{HH} J. W. Harris and S. C. Harris (2007). Survival probabilities for branching Brownian motion with absorption. {\it Elect. Comm. Probab.}, {\bf 12}, 81--92.

\bibitem{IM} K. It\^o and P. McKean (1965). \emph{Diffusion Processes and Their Sample Paths.} Springer Verlag, New York.

\bibitem{Kesten} H. Kesten (1978). Branching Brownian motion with absorption. {\it Stochastic Process. Appl.}, {\bf 7}, 9--47.

\bibitem{Maillard} P. Maillard. Branching Brownian motion with selection of the $N$ right-most particles: an approximate model. Preprint, arXiv:1112.0266.

\bibitem{RVY} B. Roynette, P. Vallois and M. Yor (2009). Penalisations of multidimensional Brownian motion, VI. {\it ESAIM: Probability and Statistics}, {\bf 13}, 152--180.

\bibitem{Weismann} A. Weismann (1889). The significance of sexual reproduction in the theory of natural selection. Pp. 251--332 in E. B. Poulton, S. Sch\"onland, and A. E. Shipley, eds., \emph{Essays upon heredity and kindred biological problems}. Clarendon Press, Oxford.

\bibitem{Williams} G.C. Williams (1966). \emph{Adaptation and natural selection.} Princeton Univ. Press, Princeton, NJ.

\end{thebibliography}
\end{document}